\documentclass[12pt]{article}

\usepackage{amsfonts}
\usepackage{amsmath}
\usepackage{amsthm}
\usepackage{amssymb}
\usepackage[shortlabels]{enumitem}
\usepackage{tikz}
\usetikzlibrary{decorations.pathreplacing}

\usepackage{xcolor}

\newcommand{\Z}{\mathbb{Z}}
\newcommand{\F}{\mathbb{F}}

\renewcommand{\SS}{\mathbb{S}}
\newcommand{\M}{\mathfrak{M}}

\newcommand{\A}{\mathfrak{A}}
\renewcommand{\aa}{\mathfrak{a}}
\newcommand{\bb}{\mathfrak{b}}
\newcommand{\cc}{\mathfrak{c}}

\newcommand{\G}{\mathcal{G}}

\newcommand{\alphah}{\widehat{\alpha}}
\newcommand{\sigmah}{\widehat{\sigma}}
\newcommand{\Kh}{\widehat{K}}

\newcommand{\OO}{\mathfrak{O}}
\newcommand{\Kt}{\widetilde{K}}

\newcommand{\Ev}{\vec{E}}
\newcommand{\omegav}{\vec{\omega}}
\newcommand{\alphav}{\vec{\alpha}}
\newcommand{\Sbar}{\overline{S}}

\DeclareMathOperator{\Gal}{Gal}
\DeclareMathOperator{\Aut}{Aut}
\DeclareMathOperator{\ch}{char}

\DeclareMathOperator{\id}{id}

\newtheorem{theorem}{Theorem}[section]
\newtheorem{proposition}[theorem]{Proposition}
\newtheorem{lemma}[theorem]{Lemma}
\newtheorem{corollary}[theorem]{Corollary}

\theoremstyle{definition}
\newtheorem{remark}[theorem]{Remark}
\newtheorem{definition}[theorem]{Definition}
\newtheorem{example}[theorem]{Example}

\numberwithin{equation}{section}

\title{Galois scaffolds for extraspecial $p$-extensions
in characteristic 0}
\author{Kevin Keating \\
Department of Mathematics \\
University of Florida \\
Gainesville, FL 32611 \\
USA \\[.2cm]
{\tt keating@ufl.edu}
\and
Paul Schwartz \\
Department of Mathematical Sciences \\
Stevens Institute of Technology \\
Hoboken, NJ 07030 \\
USA \\[2mm]
{\tt pschwart@stevens.edu}}
\begin{document}

\maketitle

\begin{abstract}
Let $K$ be a local field of characteristic 0 with
residue characteristic $p$.  Let $G$ be an extraspecial
$p$-group and let $L/K$ be a totally ramified
$G$-extension.  In this paper we find sufficient
conditions for $L/K$ to admit a Galois scaffold.  This
leads to sufficient conditions for the ring of integers
$\OO_L$ to be free of rank 1 over its associated order
$\A_{L/K}$, and to stricter conditions which imply that
$\A_{L/K}$ is a Hopf order in the group ring $K[G]$.
\end{abstract}

\section{Introduction}

Let $L/K$ be a finite totally ramified extension of
local fields and set $G=\Gal(L/K)$.  The associated
order in $K[G]$ of the ring of integers $\OO_L$ of $L$
is defined to be
\[\A_{L/K}=\{\lambda\in K[G]:
\lambda(\OO_L)\subset\OO_L\}.\]
Thus $\A_{L/K}$ is an $\OO_K$-order in $K[G]$, and
$\OO_L$ is a module over $\A_{L/K}$.  It is a
classical problem to determine when $\OO_L$ is free
(necessarily of rank 1) over $\A_{L/K}$.  In the case
where $L/K$ is tamely ramified  we have
$\A_{L/K}=\OO_K[G]$, and E. Noether showed that $\OO_L$
is free over $\A_{L/K}$ \cite{No32}.  Much less is known
about this problem in the case where $L/K$ is wildly
ramified.  When $L/K$ is a ramified $C_p$-extension,
necessary and sufficient conditions for $\OO_L$ to be
free over $\A_{L/K}$ were given in \cite{BF72,BBF72} for
the case $\ch(K)=0$, and in \cite{Ai03,Le05} for the
case $\ch(K)=p$.  Sufficient conditions are known for
$\OO_L$ to be free over $\A_{L/K}$ in the cases where
$G$ is an elementary abelian $p$-group \cite{BE18},
$G\cong C_{p^2}$ and $\ch(K)=p$ \cite{BE13}, or
$G\cong C_{p^2}$ and $\ch(K)=0$ \cite{KS22}.  When
$\ch(K)=p$, \cite{EKgen} gives sufficient conditions for
$\OO_L$ to be free over $\A_{L/K}$ for an arbitrary
$p$-group $G$.

     In this paper we give sufficient conditions for
$\OO_L$ to be free over $\A_{L/K}$ in the cases where
$\ch(K)=0$ and $G=\Gal(L/K)$ is an extraspecial
$p$-group.  As an application we give sufficient
conditions for $\A_{L/K}$ to be a Hopf order.  Our
proofs are based on constructing extensions $L/K$ which
possess a Galois scaffold, as defined in \cite{BCE18}.
We focus on the case $\ch(K)=0$, but all the statements
and proofs presented here are valid, and often simpler,
when $\ch(K)=p$.

     Let $K$ be a local field of characteristic 0 with
residue characteristic $p$.  Then $K$ is complete with
respect to a discrete valuation $v_K$ which is
normalized so that $v_K(K^\times)=\Z$.  The ring of
integers of $K$ is $\OO_K=\{x\in K:v_K(x)\geq 0\}$ and
the unique maximal ideal of $\OO_K$ is
$\M_K=\{x\in K:v_K(x)\geq 1\}$.  Let $\pi_K$ be a
uniformizer for $K$.  Thus $\pi_K$ is any element of
$\OO_K$ that satisfies $v_K(\pi_K)=1$, or equivalently
$\M_K=(\pi_K)$.  We assume that the residue field
$\OO_K/\M_K$ of $K$ is a perfect field of characteristic
$p$.  Let $e_K=v_K(p)$ be the absolute
ramification index of $K$ and let $K^{sep}$ be a
separable closure of $K$.  Then $v_K$ extends uniquely
to a valuation on $K^{sep}$, which we also denote by
$v_K$.  We will often work with towers
$K_0\subset K_1\subset\dots\subset K_n$ of field
extensions of degree $p$.  In this case we usually
denote $v_{K_i},\OO_{K_i},\M_{K_i},\pi_{K_i},e_{K_i}$ by
$v_i,\OO_i,\M_i,\pi_i,e_i$.

\section{Higher ramification theory}

In this section we recall some facts about higher
ramification theory of finite separable extensions of
local fields.  We do not assume that our extensions are
Galois, but for simplicity we only consider extensions
whose degree is a power of $p$.  For more information
about higher ramification theory for extensions which
are separable, but not necessarily Galois, see
\cite{He90}, \cite{He91}, the appendix to \cite{De84},
or Chapter~13 of \cite{Hopf}.

     Let $K$ be a local field with residue
characteristic $p$ and let $L/K$ be a subextension of
$K^{sep}/K$ of degree $p^n$.  Let $L_0/K$ be the maximal
unramified subextension of $L/K$.  Let $\Gamma$ denote
the set of $K$-embeddings of $L$ into $K^{sep}$, and let
$\Gamma_0$ be the subset of $\Gamma$ consisting of
$\sigma\in\Gamma$ such that $\sigma(L_0)=L_0$ and
$\sigma$ induces the identity map on the residue field
of $L_0$.  For real $x>0$ define
\[\Gamma_x
=\{\gamma\in\Gamma_0:v_L(\gamma(\pi_L)-\pi_L)\geq x+1\}.\]
Then $\Gamma_x$ does not depend on the choice of
uniformizer $\pi_L$ for $L$.  Observe that
$\Gamma_x\subset\Gamma_y$ for $0\le y\le x$, and
$\Gamma_x=\{\id_{L}\}$ for sufficiently large $x$.  We
say that $\Gamma_x$ is the $x$th subset in the lower
ramification filtration of $\Gamma$.

\begin{lemma}
For each $x>0$, $|\Gamma_x|$ is a power of $p$.
\end{lemma}

\begin{proof}
Let $M/K$ be the Galois closure of $L/K$.  Set
$G=\Gal(M/K)$ and $H=\Gal(M/L)$.  Then $\Gamma$ can be
identified with the coset space $G/H$.  Let $x>0$ and
set
\[G(x)=\{\sigma\in G:v_L(\sigma(\pi_L)-\pi_L)
\geq x+1\}.\]
Then $H\le G(x)\le G$ and $\Gamma_x$ is identified
with $G(x)/H$.  The lemma now follows from
Lagrange's theorem.
\end{proof}

     Say that $b>0$ is a lower ramification number for
$L/K$ if $\Gamma_b\not=\Gamma_{b+\epsilon}$ for all
$\epsilon>0$.  Say that $b>0$ is a lower ramification
number with multipicity $m$ if
$|\Gamma_b|/|\Gamma_{b+\epsilon}|=p^m$.  Let $p^t$ be
the ramification index of $L/K$.  Then $L/K$ has $t$
positive lower ramification numbers, counted with
multiplicity.  The positive lower ramification numbers
of $L/K$ can be viewed either as a multiset or as a
nondecreasing sequence $b_1\le b_2\le\dots\le b_t$ of
rational numbers.  To account for the unramified part of
$L/K$ we say that $-1$ is a lower ramification number of
$L/K$ with multipicity $n-t$.  Thus $L/K$ has a total of
$n$ lower ramification numbers, counted with
multiplicity.

     We define the positive upper ramification numbers
of $L/K$ by setting $u_1=b_1$ and
$u_{i+1}=u_i+p^{-i}(b_{i+1}-b_i)$ for $1\leq i\leq t-1$.
In addition, if $t<n$ we say that $-1$ is an upper
ramification number of $L/K$ with multipicity $n-t$.
For $1\le i\le t$ let $\Gamma^{u_i}=\Gamma_{b_i}$, and
for real $x\ge0$ set
\[\Gamma^x=\begin{cases}
\Gamma^{u_1}&\text{if }0\le x\leq u_1, \\
\Gamma^{u_j}&\text{if $u_{j-1}<x\leq u_j$ for some
$2\le j\le t$}, \\
\{\id_L\}&\text{if }x>u_t.
\end{cases}\]
Let $M/K$ be a subextension of $L/K$, and let $\Delta$
be the set of $K$-embeddings of $M$ into $K^{sep}$.  We
can define ramification subsets $\Delta_x$ and
$\Delta^x$ of $\Delta$, which give us ramification
numbers for $M/K$.  For $x>0$ we have
$\Delta^x=\{\gamma|_M:\gamma\in\Gamma^x\}$ (see for
instance Theorem~13.15 in \cite{Hopf}).  It follows that
the multiset of upper ramification numbers of $M/K$ is
contained in the multiset of upper ramification numbers
of $L/K$.

     Let $L/K$ be a Galois extension of degree $p^n$ and
set $G=\Gal(L/K)$.  Then $G$ may be identified with the
set of $K$-embeddings of $L$ into $K^{sep}$, so the
definitions given above are valid with $\Gamma$ replaced
by $G$.  In this setting we find that $G_x$ and $G^x$
are normal subgroups of $G$ for all $x\ge0$.  Let
$H\trianglelefteq G$ and set $M=L^H$.  Then the upper
ramification filtration of $\Gal(M/K)\cong G/H$ is given
by $(G/H)^x=G^xH/H$ for $x\ge0$.  

     We will need the following technical facts about
ramification numbers:

\begin{proposition}
\cite[Example 1.2]{BE18}
Let $L/K$ be a totally ramified $C_p$-extension of
local fields with ramification number $b$, and let
$x\in L\smallsetminus \{0\}$.  Then
$v_L((\sigma-1)x)\geq v_L(x)+b$, with equality if and
only if $p\nmid v_L(x)$.
\end{proposition}

\begin{proposition} \label{BigBreak}
Let $L/K$ be a totally ramified Galois extension of
degree $p^n$, with upper ramification numbers
$u_1\le\dots\le u_{n-2}\le u_{n-1}<u_n$.  Set
$G=\Gal(L/K)$, and suppose there is $H\le Z(G)$ such
that $H\cong C_p^2$.  Let $M=L^H$ be the fixed field of
$H$, and assume that $u_n$ is not an upper ramification
number of $M/K$.  Then there is a unique
$C_p$-subextension $F_0/M$ of $L/M$ such that $F_0/K$
has ramification numbers $u_1,\dots,u_{n-2},u_{n-1}$.
For all other $C_p$-subextensions $F/M$ of $L/M$, $u_n$
is an upper ramification number of $F/K$.
\end{proposition}

\begin{proof}
Since $u_n$ is not an upper ramification number of $M/K$
we have $G^{u_n}H/H=(G/H)^{u_n}=\{1\}$, and hence
$G^{u_n}\le H$.  Let $A\leq H$ be such that $|A|=p$.
Then $(G/A)^{u_n+\epsilon}=G^{u_n+\epsilon}A/A$ is
necessarily trivial, and $(G/A)^{u_n}=G^{u_n}A/A$ is
trivial if and only if $G^{u_n}\le A$.  Since
$|G^{u_n}|=p$ it follows that $u_n$ is an upper
ramification number of $L^A/K$ if and only if
$A\not=G^{u_n}$.  Set $F_0=L^{G^{u_n}}$.  Then $u_n$ is
not an upper ramification number of $F_0/K$, so the
upper ramification numbers of $F_0/K$ are
$u_1,\dots,u_{n-2},u_{n-1}$.
\end{proof}

\begin{proposition} \label{decreasing}
Let $L/K$ be a totally ramified extension of degree
$p^n$.  Let $b_1\le\dots\le b_n$ be the lower
ramification numbers of $L/K$ and let
$u_1\le\dots\le u_n$ be the upper ramification numbers.
Then for $1\leq i\leq n-1$ and $m\geq i$ we have
$b_{i+1}-b_i\le p^m(u_{i+1}-u_i)$.  
\end{proposition}

\begin{proof}
In fact $b_{i+1}-b_i=p^i(u_{i+1}-u_i)\le
p^m(u_{i+1}-u_i)$.
\end{proof}

\begin{corollary} \label{upper vs lower}
Let $L/K$ be a totally ramified extension of degree
$p^n$.  Let $b_1\le\dots\le b_n$ be the lower
ramification numbers of $L/K$ and let
$u_1\le\dots\le u_n$ be the upper ramification numbers.
Then
\begin{enumerate}[(a)]
\item $b_j-b_i\leq p^{j-1}(u_j-u_i)$ for
$1\leq i\leq j\leq n$.
\item $b_j\le p^{j-1}u_j$ for $1\le j\le n$, with
equality only if $j=1$.
\end{enumerate}
\end{corollary}

\begin{proof}
(a) The claim is clear when $i=j$.  Now assume
$1\leq i<j\leq n$.  Proposition~\ref{decreasing} gives
$b_{h+1}-b_h\le p^{j-1}(u_{h+1}-u_h)$ for
$i\leq h\leq j-1$.  Thus
\begin{align*}
    b_j-b_i & = \sum_{h=i}^{j-1}\,(b_{h+1}-b_h) \\
    & \leq \sum_{h=i}^{j-1}\,p^{j-1}(u_{h+1}-u_h) \\
    & = p^{j-1}(u_j-u_i).
\end{align*}
(b) This follows from (a) by setting $i=1$ and noting
that $b_1=u_1$.
\end{proof}

\section{Constructing abelian $p$-extensions}

Let $K$ be a local field of characteristic 0 whose
residue field has characteristic $p$.  In this section
we use Artin-Schreier polynomials and Witt addition
polynomials to construct totally ramified abelian
$p$-extensions $L/K$.  These results will be used in
Sections~\ref{hi de he} and \ref{hi met section}.  The
corresponding constructions for local fields of
characteristic $p$ are well-known.

     We define the Artin-Schreier polynomial
$\wp(X)\in\F_p[X]$ by $\wp(X)=X^p-X$.  Given $a\in K$
with $v_K(a)>-pe_K/(p-1)$, choose $\alpha\in K^{sep}$
such that $\wp(\alpha)=a$.  Let $\G_K=\Gal(K^{sep}/K)$
be the absolute Galois group of $K$ and let
$\sigma\in\G_K$.  By \cite[Theorem~5]{MW56} there is
$m_0\in\Z$ such that
$(\sigma-1)\alpha\equiv m_0\pmod{\M_K}$.  Define
$m_{\sigma}$ to be the image of $m_0$ in
$\F_p\cong\Z/p\Z$.  Then $m_{\sigma}$ does not depend on
the choice of $\alpha$, so we may define
$\chi_a^K:\G_K\to\F_p$ by $\chi_a^K(\sigma)=m_{\sigma}$.
Then $\chi_a^K$ is a continuous homomorphism.

\begin{proposition} \label{Cp extension}
Let $a\in K$ satisfy $-pe_K/(p-1)<v_K(a)<0$ and
$p\nmid v_K(a)$.  Choose $\alpha\in K^{sep}$ such that
$\wp(\alpha)=a$ and set $L=K(\alpha)$.  Then $L/K$ is a
totally ramified $C_p$-extension with upper and lower
ramification number $-v_K(a)$.  Set
$H=\ker\chi_a^K\leq\G_K$.  Then $L=(K^{sep})^H$ is the
fixed field of $H$.
\end{proposition}

\begin{proof}
Set $\wp(K)=\{\wp(b):b\in K\}$.  Our hypotheses on $a$
imply that $a\not\in\wp(K)$, so $\chi_a^K$ is
nontrivial.  Hence by \cite[Theorem~5]{MW56}, $L/K$ is a
$C_p$-extension.  Furthermore, there is
$\tau\in\Gal(L/K)$ such that
$(\tau-1)\alpha\equiv1\pmod{\M_L}$.  Since
$p$ does not divide $v_K(a)=pv_K(\alpha)$ we see that
$L/K$ is totally ramified.  In addition, it follows from
Proposition~2.5 of \cite[III]{FV} that the ramification
number of $L/K$ is $-v_L(\alpha)=-v_K(a)$.  It is clear
from Galois theory that $L=(K^{sep})^H$.
\end{proof}

     The following is proved as Lemma~3.5 in
\cite{KS22}:

\begin{lemma} \label{keatings lemma}
Let $a\in K\smallsetminus \wp(K)$ be such that
$u=-v_K(a)$ satisfies $0<u<pe_K/(p-1)$. Let
$\alpha$ be a root of $f(X)=X^p-X-a$ and set
$L=K(\alpha)$.
\begin{enumerate}[(a)]
\item If $L/K$ is a totally ramified $C_p$-extension
then there is $\sigma\in \Gal(L/K)$ such that
$\sigma(\alpha)=\alpha+1+ \epsilon$ for some
$\epsilon\in L$ with
$v_K(\epsilon)\geq e_K-(1-p^{-1})u$.
\item If $p\nmid u$, then $L/K$ is a totally ramified
$C_p$-extension and
$v_K(\epsilon)=e_K-(1-p^{-1})u$.
\end{enumerate}
\end{lemma}

\begin{proposition} \label{Pequal}
Let $L$ be a local field with residue characteristic $p$
and let $\alpha\in L$ satisfy $v_L(\alpha)=-u$, with
$0<u<e_L/(p-1)$.  Let $\beta\in L$ be such that
$\wp(\alpha)\equiv\wp(\beta)\pmod{\M_L}$.  Then there is
$k\in\Z$ such that $\beta\equiv\alpha+k\pmod{\M_L}$.
\end{proposition}

\begin{proof}
Set $\delta=\beta-\alpha$.  Then
\begin{equation} \label{Pbeta}
\wp(\beta)-\wp(\alpha)=-\delta+\sum_{i=1}^p
\binom{p}{i}\alpha^{p-i}\delta^i.
\end{equation}
If $v_L(\delta)<0$ then
$v_L(\delta^p)<v_L(\binom{p}{i}\alpha^{p-i}\delta^i)$ for
$1\le i\le p-1$ and $v_L(\delta^p)<v_L(\delta)$.  This
implies $v_L(\wp(\beta)-\wp(\alpha))=v_L(\delta^p)<0$, a
contradiction.  Hence $v_L(\delta)\ge0$.  It now follows
from (\ref{Pbeta}) and the assumption on $v_L(\alpha)$
that $0\equiv-\delta+\delta^p\pmod{\M_L}$.  Thus
$\delta\equiv k\pmod{\M_L}$ for some $k\in\Z$.
\end{proof}

     We will make frequent use of the following result:

\begin{proposition} \label{chisum}
Let $a,b\in K$ satisfy $v_K(b)\ge v_K(a)>-pe_K/(p-1)$
and $v_K(p^pa^{p-1}b)>0$.  Then
\begin{enumerate}[(a)]
\item $\chi_{a+b}^K=\chi_a^K+\chi_b^K$.
\item For each root $\alpha$ of $X^p-X-a$ and each root
$\beta$ of $X^p-X-b$ there is a unique root $\gamma$ of
$X^p-X-(a+b)$ such that
$\gamma\equiv\alpha+\beta\pmod{p\alpha^{p-1}\beta}$.
Furthermore, we have $\gamma\in K(\alpha,\beta)$.
\end{enumerate}
\end{proposition}

\begin{proof}
(a) This is proved in Proposition~5 of \cite{MW56}.
\\[\medskipamount]
(b) By the proof of Proposition~5 in \cite{MW56} there
is a unique root $\gamma$ of $X^p-X-(a+b)$ such that
$\delta:=\alpha+\beta-\gamma$ satisfies $v_K(\delta)>0$.
As in the proof of Proposition~\ref{Pequal} we get
\begin{alignat*}{2}
-\delta+\sum_{i=1}^p\binom{p}{i}\gamma^{p-i}\delta^i
&=\wp(\alpha+\beta)-\wp(\gamma) \\
&\equiv\wp(\alpha)+\wp(\beta)-(a+b)
&&\pmod{p\alpha^{p-1}\beta} \\
&\equiv0&&\pmod{p\alpha^{p-1}\beta}.
\end{alignat*}
Hence $\delta\equiv0\pmod{p\alpha^{p-1}\beta}$.  It
follows from (a) that $\gamma\in K(\alpha,\beta)$.
\end{proof}

\begin{corollary} \label{refine}
Let $K,a,b,\alpha,\beta,\gamma$ satisfy the conditions
of Proposition~\ref{chisum}.  Then for all
$\sigma\in\Gal(K(\alpha,\beta)/K)$ we have
\[(\sigma-1)(\gamma)\equiv(\sigma-1)(\alpha)
+(\sigma-1)(\beta)\pmod{p\alpha^{p-1}}.\]
\end{corollary}

\begin{proof}
By Lemma~\ref{keatings lemma} there are $k,\ell,m\in\Z$
such that
\begin{alignat*}{2}
(\sigma-1)(\alpha)&\equiv k&&\pmod{p\alpha^{p-1}} \\
(\sigma-1)(\beta)&\equiv\ell&&\pmod{p\beta^{p-1}} \\
(\sigma-1)(\gamma)&\equiv m&&\pmod{p\gamma^{p-1}}.  
\end{alignat*}
Since $\chi_{a+b}^K(\sigma)=\chi_a^K(\sigma)
+\chi_b^K(\sigma)$ we have $m\equiv k+\ell\pmod{p}$.
Since $v_K(\alpha)\le v_K(\beta)$ and
$v_K(\alpha)\le v_K(\gamma)$, the corollary follows from
this.
\end{proof}

\begin{corollary} \label{congruent}
Let $a,a'\in K$ satisfy $v_K(a)>-pe_K/(p-1)$ and
$a\equiv a'\pmod{\M_K}$.  Then
\begin{enumerate}[(a)]
\item $\chi_a^K=\chi_{a'}^K$.
\item For each root $\alpha$ of $X^p-X-a$ there is a
root $\alpha'$ of $X^p-X-a'$ such that
$v_K(\alpha-\alpha')>0$.
\end{enumerate}
\end{corollary}

\begin{proof}
We have $a'=a+b$ with $v_K(b)>0$.  Since $\chi_b$ is the
trivial character, the corollary follows.
\end{proof}

\begin{proposition} \label{equalbreaks}
Let $u<e_K$ be a positive integer.  Choose
$a_1,\dots,a_r\in\M_K^{-u}$ such that
\[\Sbar=\{a_i+\M_K^{-u+1}:1\leq i\leq r\}\]
is an $\F_p$-linearly independent subset of
$\M_K^{-u}/\M_K^{-u+1}$.  For $1\le i\le r$ let
$\alpha_i\in K^{sep}$ be a root of $X^p-X-a_i$.  Then
$L=K(\alpha_1,\dots,\alpha_r)$ is an elementary abelian
$p$-extension of $K$ of rank $r$, with a single upper
and lower ramification number $u$ of multiplicity $r$.
\end{proposition}

\begin{proof}
Clearly $L/K$ is an elementary abelian $p$-extension of
degree $p^s$ for some $s\le r$.  Let
$n_1,\dots,n_r\in\Z$ and set $b=n_1a_1+\dots+n_ra_r$.
Then $\chi_b^K=n_1\chi_{a_1}^K+\dots+n_r\chi_{a_r}^K$ by
Proposition~\ref{chisum}(a) and the bound on $u$.
Therefore $\chi_b^K$ induces a character from
$\Gal(L/K)$ to $\F_p$.  If $n_1',\dots,n_r'$ are
integers such that $b'=n_1'a_1+\dots+n_r'a_r$ satisfies
$\chi_{b'}^K=\chi_b^K$ then $\chi_{b'-b}^K$ is trivial.
Hence $b'\equiv b\pmod{\M_K^{-u}}$, so by the
independence of $\Sbar$ we get $n_i'\equiv n_i\pmod{p}$
for $1\le i\le r$.  It follows that there are $p^r$
distinct characters from $\Gal(L/K)$ to $\F_p$, so we
must have $[L:K]=p^r$.  If $\chi_b^K$ is nontrivial then
$p\nmid n_i$ for at least one $i$, so $v_K(b)=-u$ by the
independence of $\Sbar$.  Hence the $C_p$-subextension
of $L/K$ corresponding to $\chi_b^K$ has upper
ramification number $u$.  Therefore the only upper
ramification number of $L/K$ is $u$, which has
multiplicity $r$.  It follows that the only lower
ramification number of $L/K$ is $u$, again with
multiplicity $r$.
\end{proof}

     Motivated by these results we make the following
definition:

\begin{definition}
Let $K$ be a local field with residue characteristic
$p$.  We say $a_1,\dots,a_n\in K$ are reduced
Artin-Schreier constants if all of the following hold:
\begin{enumerate}[(i)]
\item $-pe_K/(p-1)<v_K(a_n)\leq\dots\leq
v_K(a_1)<0$.
\item $p\nmid v_K(a_i)$ for each $1\leq i\leq n$.
\item If $v_K(a_i)=v_K(a_{i+1})=\dots=v_K(a_j)=-u$
then $\{a_h+\M_K^{-u+1}:i\leq h\leq j\}$ is an
$\F_p$-linearly independent subset of
$\M_K^{-u}/\M_K^{-u+1}$.
\item $v_K\left(p^p a_n^{p-1}a_{n-1}\right)>0$.
\end{enumerate}
\end{definition}

\begin{proposition} \label{ele ab}
Let $a_1,\dots,a_n\in K$ be reduced Artin-Schreier
constants.  For $1\leq i\leq n$ set $u_i=-v_K(a_i)$ and
let $\alpha_i\in K^{sep}$ be a root of $X^p-X-a_i$.  Set
$L=K(\alpha_1,\dots,\alpha_n)$.  Then $L/K$ is a totally
ramified $C_p^n$-extension with upper ramification
numbers $u_1\le\dots\le u_n$.
\end{proposition}

\begin{proof}
Let $v_1<\dots<v_t$ be the distinct elements of the
multiset $\{u_i:1\le i\le n\}$ and let $m_j$ be the
multiplicity of $v_j$.  It follows from
Proposition~\ref{equalbreaks} that for $1\le j\le t$,
$L/K$ has a subextension $L_j/K$ which is elementary
abelian of rank $m_j$ and whose only upper ramification
number is $v_j$, with multiplicity $m_j$.  It follows
that $L_1L_2\dots L_t/K$ has upper ramification numbers
$v_1,\dots,v_t$ with multiplicities $m_1,\dots,m_t$.
Since $m_1+\dots+m_t=n$ we get $L=L_1L_2\dots L_t$.
\end{proof}

     It follows from Lemma~\ref{keatings lemma} that for
each $1\le i\le n$ there is
$\gamma_i\in\Gal(K(\alpha_i)/K)$ such that
$\gamma_i(\alpha_i)=\alpha_i+1+\delta_i$, where
$v_K(\delta_i)=v_K(p\alpha_i^{p-1})>0$.  We may extend
$\gamma_i$ to an automorphism of $L$ by setting
\begin{equation} \label{can gen}
\gamma_i(\alpha_j)=\begin{cases}
\alpha_j & \mbox{ if }  j\neq i, \\
\alpha_i+1+\delta_i & \mbox{ if }  j=i.
\end{cases}
\end{equation}

     Define
\[D(X,Y)=\frac{X^p+Y^p-(X+Y)^p}{p}=
-\sum_{i=1}^{p-1}p^{-1}\binom{p}{i}X^iY^{p-i}.\]
Then $D(X,Y)\in\Z[X,Y]$ and
\[S_1(X_0,X_1,Y_0,Y_1)=X_1+Y_1+D(X_0,Y_0)\]
is the second Witt addition polynomial.

\begin{proposition} \label{Cp2}
Let $a_1,a_2\in K$ and set $u_i=-v_K(a_i)$ for $i=1,2$.
Assume that $p\nmid u_1$, $(p-p^{-1})u_1<e_K$, and
$(1-p^{-1}+p^{-2})u_1+(1-p^{-1})u_2<e_K$.  Let
$\alpha_1,\alpha_2\in K^{sep}$ satisfy
$\alpha_1^p-\alpha_1=a_1$ and
$\alpha_2^p-\alpha_2=D(\alpha_1,a_1)+a_2$.  Set
$E=K(\alpha_1,\alpha_2)$.  Then
\begin{enumerate}[(a)]
\item $E/K$ is a $C_{p^2}$-extension.
\item If $u_2>pu_1$ and $p\nmid u_2$ then the upper
ramification numbers of $E/K$ are $u_1,u_2$.
\item If $u_2>(p+1-p^{-1})u_1$ and $p\nmid u_2$ then
there is a generator $\theta$ for $\Gal(E/K)$ such that
$\theta(\alpha_1)\equiv\alpha_1+1\pmod{p\alpha_1^{p-1}}$
and $\theta(\alpha_2)\equiv\alpha_2+D(1,\alpha_1)
\pmod{\M_E}$.
\end{enumerate}
\end{proposition}

\begin{proof}
For (a) see the remark in Section~3 of \cite{VZ95}, or
Theorem~2.1 in \cite{KS22}.  For (b) see Proposition~3.4
in \cite{KS22}, and for (c) see Proposition~3.6 in
\cite{KS22}.  Propositions~3.4 and 3.6 in \cite{KS22}
are proved under the hypothesis that $a_2=a_1\mu^p$ for
some $\mu\in K^{\times}$.  However, this assumption is
not actually used in the proofs.
\end{proof}

\section{Generalized Heisenberg extensions}
\label{hi de he}

There are two families of extraspecial $p$-groups: The
generalized Heisenberg groups, which have exponent $p$,
and the generalized metacyclic groups, which have
exponent $p^2$.  In this section we define the
generalized Heisenberg group $H(n)$ of order $p^{2n+1}$.
We then show how to construct totally ramified
$H(n)$-extensions of local fields with residue
characteristic $p$.  In Section~\ref{hi met section} we
will define the generalized metacyclic groups $M(n)$ and
construct totally ramified $M(n)$-extensions.  In
Section~\ref{build} we will use the results of this
section to construct generalized Heisenberg extensions
which have a Galois scaffold.

\begin{definition} \label{heis def}
Let $p$ be an odd prime.  For $n\ge1$ we define a group
$H(n)$ of order $p^{2n+1}$ generated by
$g_1,\ldots,g_{2n},g_{2n+1}$, with $|g_i|=p$ for
$1\le i\le2n+1$.  All these generators commute with each
other, except for $g_i$ and $g_{n+i}$, which satisfy
$[g_i,g_{n+i}]=g_{2n+1}$ for $1\le i\le n$.  Thus $H(1)$
is the Heisenberg $p$-group and $H(n)$ is an
extraspecial $p$-group with exponent $p$.
\end{definition}

\begin{proposition} \label{hi he}
Let $K_0$ be a local field with residue characteristic
$p>2$.  Let $a_1,\dots,a_{2n},a_{2n+1}$ be elements of
$K_0$ such that $a_1,\dots,a_{2n}$ are reduced
Artin-Schreier constants.  Set $u_i=-v_0(a_i)$ for
$1\le i\le2n+1$ and assume that
\begin{align}
u_n+(1-p^{-1})u_{2n}&<e_0, \label{hineq1} \\
p^{-1}u_n+p^{-2}u_{2n}+(1-p^{-1})u_{2n+1}&<e_0.
\label{hineq3}
\end{align}
For $1\le i\le2n$ let $\alpha_i$ satisfy
$\alpha_i^p-\alpha_i=a_i$ and define $K_1,\dots,K_{2n}$
recursively by $K_i=K_{i-1}(\alpha_i)$.  Set
\[B=a_1\alpha_{n+1}+a_2\alpha_{n+2}+\dots
+a_n\alpha_{2n},\]
let $\alpha_{2n+1}$ satisfy
$\alpha_{2n+1}^p-\alpha_{2n+1}=B+a_{2n+1}$, and define
$K_{2n+1}=K_{2n}(\alpha_{2n+1})$.  Then
\begin{enumerate}[(a)]
\item $K_{2n+1}/K_0$ is a totally ramified
$H(n)$-extension.  In addition, for $1\le i\le2n+1$
there are $\sigma_i\in\Gal(K_{2n+1}/K_{i-1})$ such that
$\sigma_i|_{K_i}$ generates $\Gal(K_i/K_{i-1})$, with
the following properties:
\begin{alignat}{3}
\sigma_i(\alpha_j)&=\alpha_j&&&&
\text{ for $1\le i<j\le2n$}, \label{hshift1} \\
\sigma_i(\alpha_i)&\equiv\alpha_i+1
&&\pmod{p\alpha_i^{p-1}}&&\text{ for $1\le i\le2n+1$},
\label{hshift2} \\
\sigma_i(\alpha_{2n+1})&=\alpha_{2n+1}
&&&&\text{ for $1\le i\le n$}, \label{hshift3} \\
\sigma_{n+i}(\alpha_{2n+1})&\equiv
\alpha_{2n+1}+\alpha_i&&\pmod{\M_{2n+1}}\;
&&\text{ for $1\le i\le n$}, \label{hshift4}.
\end{alignat}
\item If $u_{2n+1}\le u_n+u_{2n}$ then there is
$v\le u_n+u_{2n}$ such that the upper ramification
numbers of $K_{2n+1}/K_0$ are $u_1,\dots,u_{2n},v$ (but
not necessarily in that order).
\item Suppose $u_{2n+1}>u_n+u_{2n}$ and
$p\nmid u_{2n+1}$.  Then the upper ramification numbers
of $K_{2n+1}/K_0$ are $u_1\le\dots\le u_{2n}<u_{2n+1}$,
and $v_0(\alpha_i)=-p^{-1}u_i$ for $1\le i\le2n+1$.
\end{enumerate}
\end{proposition}

\begin{proof}
(a) It follows from Proposition~\ref{ele ab} that
$K_{2n}/K_0$ is a totally ramified $C_p^{2n}$-extension
with upper ramification numbers $u_1,\dots,u_{2n}$.  For
$1\le i\le 2n$ let $\gamma_i$ be the element of
$\Gal(K_{2n}/K_0)$ defined by (\ref{can gen}).  Let
$1\le i\le n$.  Then $\gamma_i(B+a_{2n+1})=B+a_{2n+1}$,
and by Lemma~\ref{keatings lemma} we have
\begin{alignat}{2} \nonumber
\gamma_{n+i}(B+a_{2n+1})-(B+a_{2n+1})
&=a_i(\gamma_{n+i}(\alpha_{n+i})-\alpha_{n+i}) \\
&\equiv a_i&&\pmod{pa_i\alpha_{n+i}^{p-1}}.
\label{tauB}
\end{alignat}
Hence by (\ref{hineq1}) we get
\begin{equation} \label{gammaB}
\gamma_{n+i}(B+a_{2n+1})\equiv B+a_{2n+1}+a_i
\pmod{\M_{2n}}.
\end{equation}

     Suppose $\alpha_{2n+1}\in K_{2n}$, so that
$K_{2n+1}=K_{2n}$.  Then for $1\le i\le n$,
\[\wp(\gamma_i(\alpha_{2n+1}))=
\gamma_i(\wp(\alpha_{2n+1}))=\gamma_i(B+a_{2n+1})
=B+a_{2n+1}=\wp(\alpha_{2n+1}).\]
Furthermore, by (\ref{gammaB}) and (\ref{hineq3}) we get
\begin{alignat*}{2}
\wp(\gamma_{n+i}(\alpha_{2n+1}))&=\gamma_{n+i}(B+a_{2n+1}) \\
&\equiv B+a_{2n+1}+a_i&&\pmod{\M_{2n}} \\
&\equiv\wp(\alpha_{2n+1})+\wp(\alpha_i)&&\pmod{\M_{2n}} \\
&\equiv\wp(\alpha_{2n+1}+\alpha_i)&&\pmod{\M_{2n}}.
\end{alignat*}
Hence by Proposition~\ref{Pequal} there are
$k,\ell\in\Z$ such that
\begin{alignat*}{2}
\gamma_i(\alpha_{2n+1})&\equiv \alpha_{2n+1}+k
&&\pmod{\M_{2n}} \\
\gamma_{n+i}(\alpha_{2n+1})&\equiv
\alpha_{2n+1}+\alpha_i+\ell&&\pmod{\M_{2n}}.
\end{alignat*}
It follows that
\begin{alignat*}{2}
\gamma_i\gamma_{n+i}(\alpha_{2n+1})&\equiv
\gamma_i(\alpha_{2n+1}+\alpha_i+\ell)&&\pmod{\M_{2n}} \\
&\equiv \alpha_{2n+1}+\alpha_i+1+k+\ell&&\pmod{\M_{2n}}
\\[2mm]
\gamma_{n+i}\gamma_i(\alpha_{2n+1})&\equiv
\gamma_{n+i}(\alpha_{2n+1}+k)&&\pmod{\M_{2n}} \\
&\equiv \alpha_{2n+1}+\alpha_i+k+\ell&&\pmod{\M_{2n}}.
\end{alignat*}
Since $\gamma_i\gamma_{n+i}=\gamma_{n+i}\gamma_i$ this
is a contradiction.  It follows that
$\alpha_{2n+1}\not\in K_{2n}$.

     By (\ref{hineq1}) and (\ref{hineq3}) we have
$v_0(B+a_{2n+1})>-pe_0/(p-1)$.  Hence by
Proposition~\ref{Cp extension} $K_{2n+1}/K_{2n}$ is a
$C_p$-extension, and we can choose
$\sigma_{2n+1}\in\Gal(K_{2n+1}/K_{2n})$ so that
\begin{equation} \label{x2n1}
\sigma_{2n+1}(\alpha_{2n+1})-\alpha_{2n+1}
\equiv1\pmod{\M_{2n+1}}.
\end{equation}
Let $1\le i\le n$.  By (\ref{hineq1}), (\ref{hineq3}),
and Proposition~\ref{chisum}(b),
$X^p-X-(B+a_{2n+1}+a_i)$ has a root $\delta$ such that
$v_0(\alpha_{2n+1}+\alpha_i-\delta)>0$.  It follows from
(\ref{gammaB}) that there is $a_i'\in K_{2n}$ such that
$\gamma_{n+i}(B+a_{2n+1})=B+a_{2n+1}+a_i'$ and
$a_i'\equiv a_i\pmod{\M_{2n}}$.  Hence by
Corollary~\ref{congruent}(b) we see that
$X^p-X-\gamma_{n+i}(B+a_{2n+1})$ has a root
$\alpha_{2n+1}'$ such that
$v_0(\alpha_{2n+1}+\alpha_i-\alpha_{2n+1}')>0$.  By
Corollary~\ref{congruent}(a) and
Proposition~\ref{chisum}(a) we get
\[\chi_{\gamma_{n+i}(B+a_{2n+1})}^{K_{2n}}
=\chi_{B+a_{2n+1}+a_i'}^{K_{2n}}
=\chi_{B+a_{2n+1}+a_i}^{K_{2n}}
=\chi_{B+a_{2n+1}}^{K_{2n}}+\chi_{a_i}^{K_{2n}}.\]
Since $a_i=\wp(\alpha_i)$ it follows that
$\chi_{\gamma_{n+i}(B+a_{2n+1})}^{K_{2n}}
=\chi_{B+a_{2n+1}}^{K_{2n}}$, and hence that
$\alpha_{2n+1}'\in K_{2n+1}$.  Therefore we may extend
$\gamma_{n+i}$ to $\sigma_{n+i}\in\Aut(K_{2n+1})$ by
setting $\sigma_{n+i}(\alpha_{2n+1})=\alpha_{2n+1}'$.
For $1\le i\le n$ we have $\gamma_i(B)=B$, so we may
extend $\gamma_i$ to $\sigma_i\in\Aut(K_{2n+1})$ by
setting $\sigma_i(\alpha_{2n+1})=\alpha_{2n+1}$.  Hence
$K_{2n+1}/K_0$ is a Galois extension.  

     Statements (\ref{hshift1}), (\ref{hshift3}), and
(\ref{hshift4}) follow directly from the definitions of
$\sigma_i$ and $\sigma_{n+i}$, while (\ref{hshift2})
follows from the definitions and
Lemma~\ref{keatings lemma}.  Since $K_{2n}/K_0$ is an
elementary abelian $p$-extension we have
$[\sigma_i,\sigma_j]\in\Gal(K_{2n+1}/K_{2n})$ and
$\sigma_i^p\in\Gal(K_{2n+1}/K_{2n})$ for
$1\le i,j\le2n+1$.  By considering how these elements
act on $\alpha_{2n+1}$ we easily find that
$\sigma_i^p=1$ for $1\le i\le2n+1$, and
$[\sigma_i,\sigma_j]=1$ for
$1\le i\le j\le2n+1$ unless $1\le i\le n$ and $j=n+i$.
For instance, if $1\le i<j\le n$ then
\begin{alignat*}{2}
[\sigma_{n+i},\sigma_{n+j}](\alpha_{2n+1})
&=\sigma_{n+i}\sigma_{n+j}\sigma_{n+i}^{-1}
\sigma_{n+j}^{-1}(\alpha_{2n+1}) \\
&\equiv\sigma_{n+i}\sigma_{n+j}\sigma_{n+i}^{-1}
(\alpha_{2n+1}-\alpha_j)&&\pmod{\M_{2n+1}} \\
&\equiv\sigma_{n+i}\sigma_{n+j}
(\alpha_{2n+1}-\alpha_i-\alpha_j)&&\pmod{\M_{2n+1}} \\
&\equiv\sigma_{n+i}(\alpha_{2n+1}-\alpha_i)
&&\pmod{\M_{2n+1}} \\
&\equiv \alpha_{2n+1}&&\pmod{\M_{2n+1}}.
\end{alignat*}
Hence $[\sigma_{n+i},\sigma_{n+j}](\alpha_{2n+1})
=\alpha_{2n+1}$, so $[\sigma_{n+i},\sigma_{n+j}]=1$.
On the other hand, using (\ref{x2n1}) we get
\begin{alignat*}{2}
[\sigma_i,\sigma_{n+i}](\alpha_{2n+1})
&=\sigma_i\sigma_{n+i}\sigma_i^{-1}
\sigma_{n+i}^{-1}(\alpha_{2n+1}) \\
&\equiv\sigma_i\sigma_{n+i}\sigma_i^{-1}
(\alpha_{2n+1}-\alpha_i)&&\pmod{\M_{2n+1}} \\
&\equiv\sigma_i\sigma_{n+i}(\alpha_{2n+1}-\alpha_i+1)
&&\pmod{\M_{2n+1}} \\
&\equiv\sigma_i(\alpha_{2n+1}+1)&&\pmod{\M_{2n+1}} \\
&\equiv \alpha_{2n+1}+1&&\pmod{\M_{2n+1}} \\
&\equiv\sigma_{2n+1}(\alpha_{2n+1})&&\pmod{\M_{2n+1}}.
\end{alignat*}
Therefore $[\sigma_i,\sigma_{n+i}]=\sigma_{2n+1}$ for
$1\le i\le n$.  We conclude that
$\Gal(K_{2n+1}/K_0)\cong H(n)$. \\[\medskipamount]
(b) For $1\le i\le n$ let $L_{(i)}$ be the extension of
$K_0(\alpha_i,\alpha_{n+i})$ generated by the roots of
$X^p-X-a_i\alpha_{n+i}$.  It follows from (a) that
$L_{(i)}/K_0$ is an $H(1)$-extension.  Let $M_{(i)}$ be
the extension of $K_0(\alpha_{n+i})$ generated by the
roots of $X^p-X-a_i\alpha_{n+i}$.  Then
$M_{(i)}/K_0(\alpha_{n+i})$ is a $C_p$-extension with
ramification number
$-v_1(a_i\alpha_{n+i})=pu_i+u_{n+i}$.  Since
$K_0(\alpha_{n+i})/K_0$ is a $C_p$-extension with
ramification number $u_{n+i}$, $M_{(i)}/K_0$ is an
extension of degree $p^2$ with lower ramification
numbers $u_{n+i},pu_i+u_{n+i}$.  Hence the upper
ramification numbers of $M_{(i)}/K_0$ are
$u_{n+i},u_i+u_{n+i}$.  Since
$K_0(\alpha_i,\alpha_{n+i})/K_0$ has upper ramification
numbers $u_i,u_{n+i}$, and
$K_0(\alpha_i,\alpha_{n+i})\subset L_{(i)}$, we deduce
that the upper ramification numbers of $L_{(i)}/K_0$ are
$u_i,u_{n+i},u_i+u_{n+i}$.

     Set $L=L_{(1)}L_{(2)}\dots L_{(n)}$.  Then
$K_{2n}\subset L$.  Let $\beta\in K_0^{sep}$ be a root of
$X^p-X-a_{2n+1}$.  Then $L(\beta)/K_0$ is a compositum
of Galois $p$-extensions, so $L(\beta)/K_0$ is a Galois
$p$-extension.  Set $G=\Gal(L(\beta)/K_0)$ and
$N_{(i)}=\Gal(L(\beta)/L_{(i)})$ for $1\le i\le n$.
Then $N_{(i)}\trianglelefteq G$ and
$\Gal(L_{(i)}/K_0)\cong G/N_{(i)}$.  For
$x>u_i+u_{n+i}$ we have $G^xN_{(i)}/N_{(i)}
=(G/N_{(i)})^x=\{1\}$, and hence $G^x\le N_{(i)}$.  Set
$H=\Gal(L(\beta)/K_0(\beta))$; then $H\trianglelefteq G$
and $\Gal(K_0(\beta)/K_0)\cong G/H$.  For
$x>u_n+u_{2n}\ge u_{2n+1}$ we get $G^xH/H=(G/H)^x=\{1\}$,
and hence $G^x\le H$.  Since
$N_{(1)}\cap\dots\cap N_{(n)}\cap H=\{1\}$, it follows
that $G^x=\{1\}$ for all $x>u_n+u_{2n}$.  Therefore
the upper ramification numbers of $L(\beta)/K_0$ are all
$\le u_n+u_{2n}$.  By (\ref{hineq1}), (\ref{hineq3}),
and Proposition~\ref{chisum}(a) we have
\[\chi_{B+a_{2n+1}}^{K_{2n}}=
\chi_{a_1\alpha_{n+1}}^{K_{2n}}+
\dots+\chi_{a_n\alpha_{2n}}^{K_{2n}}
+\chi_{a_{2n+1}}^{K_{2n}}.\]
Hence $K_{2n+1}\subset L(\beta)$, so the upper
ramification numbers of $K_{2n+1}/K_0$ are bounded above
by $u_n+u_{2n}$.  Since the subextension $K_{2n}/K_0$
of $K_{2n+1}/K_0$ has upper ramification numbers
$u_1,\dots,u_{2n}$, we conclude that there is
$v\le u_n+u_{2n}$ such that the upper ramification
numbers of $K_{2n+1}/K_0$ are $u_1,\dots,u_{2n},v$.
\\[\medskipamount]
(c) Let $\Kt_{2n+1}$ be the extension of $K_{2n}$
generated by the roots of $X^p-X-B$.  Then by (a) and
(b) $\Kt_{2n+1}/K_0$ is an $H(n)$-extension whose upper
ramification numbers are $u_1,\dots,u_{2n},v$ for some
$v\le u_n+u_{2n}$.  Once again let $\beta\in K_0^{sep}$
be a root of $X^p-X-a_{2n+1}$.  Since $p\nmid u_{2n+1}$
we see that $u_{2n+1}$ is an upper ramification number
of $K_0(\beta)/K_0$, and hence also of
$\Kt_{2n+1}(\beta)/K_0$.  Since $u_{2n+1}>u_n+u_{2n}$,
$u_{2n+1}$ is not an upper ramification number of
$\Kt_{2n+1}/K_0$.  Therefore
$\Kt_{2n+1}(\beta)\not=\Kt_{2n+1}$ and
$\Kt_{2n+1}(\beta)/K_{2n}$ is a $C_p^2$-extension.  By
(\ref{hineq3}) and Proposition~\ref{chisum}(a) we have
$\chi_{B+a_{2n+1}}^{K_{2n}}=\chi_B^{K_{2n}}
+\chi_{a_{2n+1}}^{K_{2n}}$.  Hence
$K_{2n+1}\subset\Kt_{2n+1}(\beta)$ and
$K_{2n+1}\not=\Kt_{2n+1}$.  The upper ramification
numbers of the elementary abelian $p$-extension
$K_{2n}(\beta)/K_0$ are $u_1,\dots,u_{2n},u_{2n+1}$.
Since $u_{2n+1}>u_n+u_{2n}\ge v$, the upper ramification
numbers of $\Kt_{2n+1}(\beta)/K_0$ are
$u_1,\dots,u_{2n},u_{2n+1},v$.  Using
Proposition~\ref{BigBreak} we deduce that the upper
ramification numbers of $K_{2n+1}/K_0$ are
$u_1,\dots,u_{2n},u_{2n+1}$.  We clearly have
$v_0(\alpha_i)=p^{-1}v_0(a_i)=-p^{-1}u_i$ for
$1\le i\le2n$.  Since $u_{2n+1}>u_n+u_{2n}$ we also have
\[v_0(\alpha_{2n+1})=p^{-1}v_0(B+a_{2n+1})
=p^{-1}v_0(a_{2n+1})=-p^{-1}u_{2n+1}. \qedhere \]
\end{proof}

\section{Generalized metacyclic extensions}
\label{hi met section}

In this section we define the generalized metacyclic
group $M(n)$ of order $p^{2n+1}$.  We then show how to
construct totally ramified $M(n)$-extensions of local
fields with residue characteristic $p$.  In
Section~\ref{build} we will use these results to
construct generalized metacyclic extensions which
possess a Galois scaffold.

\begin{definition}
Let $p$ be an odd prime.  For $n\ge1$ we define a group
$M(n)$ of order $p^{2n+1}$ generated by
$h_1,\ldots,h_{2n},h_{2n+1}$, with $h_1^p=h_{2n+1}$ and
$|h_i|=p$ for $2\le i\le2n+1$.  All these generators
commute with each other, except for $h_i$ and $h_{n+i}$,
which satisfy $[h_i,h_{n+i}]=h_{2n+1}$ for
$1\le i\le n$.  Thus $M(1)$ is the metacyclic group of
order $p^3$ and $M(n)$ is an extraspecial $p$-group with
exponent $p^2$.
\end{definition}

\begin{proposition} \label{extra special p ext}
Let $K_0$ be a local field with residue characteristic
$p>2$ and let $a_1,a_2,\dots,a_{2n},a_{2n+1}\in K_0$.
For $1\le i\le2n+1$ set $u_i=-v_0(a_i)$.  Assume that
$a_1,a_2,\dots,a_{2n}$ are reduced Artin-Schreier
constants and $p\nmid u_{2n+1}$.  Assume further that
\begin{align}
u_n+(1-p^{-1})u_{2n}&<e_0 \label{mineq1} \\
p^{-1}u_n+p^{-2}u_{2n}+(1-p^{-1})u_{2n+1}&<e_0
\label{mineq3} \\
(1-p^{-1}+p^{-2})u_1+(1-p^{-1})u_{2n+1}&<e_0
\label{mineq2} \\
u_n+u_{2n}&<u_{2n+1} \label{mineq4} \\
(p+1-p^{-1})u_1&<u_{2n+1}. \label{mineq5}
\end{align}
For $1\le i\le2n$ let $\alpha_i$ satisfy
$\alpha_i^p-\alpha_i=a_i$.  Set
\begin{align*}
B&=a_1\alpha_{n+1}+a_2\alpha_{n+2}+\dots+a_n\alpha_{2n}
\\
C&=D(\alpha_1,a_1)=-\sum_{i=1}^{p-1}
p^{-1}\binom{p}{i}\alpha_1^ia_1^{p-i}
\end{align*}
and let $\alphah_{2n+1}\in K_0^{sep}$ satisfy
$\alphah_{2n+1}^p-\alphah_{2n+1}=B+C+a_{2n+1}$.  Define
$K_1,\dots,K_{2n}$ recursively by
$K_i=K_{i-1}(\alpha_i)$ for $1\leq i\leq 2n$, and set
$\Kh_{2n+1}=K_{2n}(\alphah_{2n+1})$.  Then
\begin{enumerate}[(a)]
\item $\Kh_{2n+1}/K_0$ is a totally ramified extension
with upper ramification numbers
$u_1\le\dots\le u_{2n}\le u_{2n+1}$.  Furthermore, we
have $v_0(\alpha_i)=-p^{-1}u_i$ for $1\le i\le2n+1$.
\item $\Kh_{2n+1}/K_0$ is a Galois extension with Galois
group $M(n)$.
\item For $1\le i\le2n+1$ there are
$\sigmah_i\in\Gal(\Kh_{2n+1}/K_{i-1})$ such that
$\sigmah_i|_{K_i}$ generates $\Gal(K_i/K_{i-1})$ for
$1\le i\le2n$ and $\sigmah_{2n+1}$ generates
$\Gal(\Kh_{2n+1}/K_{2n})$, with the following
properties:
\begin{alignat}{3}
\sigmah_i(\alpha_j)&=\alpha_j&&&&
\text{ for $1\le i<j\le2n$}, \label{mshift1} \\
\sigmah_i(\alpha_i)&\equiv\alpha_i+1
&&\pmod{p\alpha_i^{p-1}}&&\text{ for $1\le i\le2n$},
\label{mshift2} \\
\sigmah_1(\alphah_{2n+1})&\equiv\alphah_{2n+1}
&&\pmod{\alpha_1^{p-1}},\label{mshift6} \\
\sigmah_i(\alphah_{2n+1})&=\alphah_{2n+1}
&&&&\text{ for $2\le i\le n$}, \label{mshift3} \\
\sigmah_{n+i}(\alphah_{2n+1})&\equiv
\alphah_{2n+1}&&\pmod{\alpha_i}\;
&&\text{ for $1\le i\le n$}, \label{mshift4} \\
\sigmah_{2n+1}(\alphah_{2n+1})&\equiv
\alphah_{2n+1}+1&&\pmod{p\alphah_{2n+1}^{p-1}}.
\label{mshift5}
\end{alignat}
\end{enumerate}
\end{proposition}

\begin{proof}
(a) Let $K_{2n+1}$ be the extension of $K_{2n}$
generated by the roots of $X^p-X-B$.  Thanks to
assumptions (\ref{mineq1}) and (\ref{mineq3}) we can
apply Proposition~\ref{hi he}, which says that
$K_{2n+1}/K_0$ is an $H(n)$-extension whose upper
ramification numbers are $u_1,\dots,u_{2n},v$ for some
$v\le u_n+u_{2n}$.  Let $E$ be the extension of $K_1$
generated by the roots of $X^p-X-(C+a_{2n+1})$.  Then by
(\ref{mineq2}), (\ref{mineq5}), and
Proposition~\ref{Cp2}, $E/K_0$ is a $C_{p^2}$-extension
with upper ramification numbers $u_1,u_{2n+1}$.
Furthermore, $K_{2n+1}\cap E=K_1$ and $K_{2n+1}E$ is a
$C_p^2$-extension of $K_{2n}$.  Using (\ref{mineq5}),
(\ref{mineq4}), (\ref{mineq3}), and
Proposition~\ref{chisum}(a) we get
$\chi_{B+C+a_{2n+1}}^{K_{2n}}
=\chi_{B}^{K_{2n}}+\chi_{C+a_{2n+1}}^{K_{2n}}$, and
hence $\Kh_{2n+1}\subset K_{2n+1}E$.  Therefore by
(\ref{mineq4}) and Proposition~\ref{BigBreak} the upper
ramification numbers of $\Kh_{2n+1}/K_0$ are
$u_1\le\dots\le u_{2n}\le u_{2n+1}$.  We clearly have
$v_0(\alpha_i)=p^{-1}v_0(a_i)=-p^{-1}u_i$ for
$1\le i\le2n$.  It follows from (\ref{mineq4}) and
(\ref{mineq5}) that
$v_0(B+C+a_{2n+1})=v_0(a_{2n+1})=-u_{2n+1}$.  Therefore
$v_0(\alpha_{2n+1})=-p^{-1}u_{2n+1}$.
\\[\medskipamount]
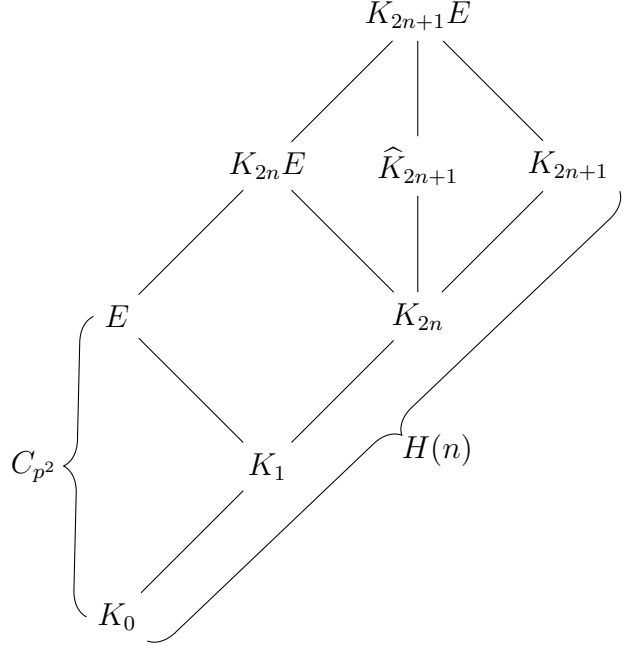
\begin{figure}
\begin{center}
\vspace{-6mm}
\begin{tikzpicture}[node distance = 2cm, auto]
\node (K0) {$K_0$};
\node (K1) [right of=K0, above of=K0] {$K_1$};
\node (E) [left of=K1, above of=K1] {$E$};
\node (K2n) [right of=K1, above of=K1] {$K_{2n}$};
\node (K2n1) [right of=K2n, above of=K2n] {$K_{2n+1}$};
\node (K2n1h) [above of=K2n] {$\Kh_{2n+1}$};
\node (K2nE) [above of=E, right of=E] {$K_{2n}E$};
\node (K2n1E) [above of=K2nE, right of=K2nE] {$K_{2n+1}E$};
\draw[-] (K1) to (E);
\draw[-] (K1) to (K2n);
\draw[-] (K0) to (K1);
\draw[-] (K2n) to (K2nE);
\draw[-] (E) to (K2nE);
\draw[-] (K2n) to (K2n1);
\draw[-] (K2n) to (K2n1h);
\draw[-] (K2nE) to (K2n1E);
\draw[-] (K2n1) to (K2n1E);
\draw[-] (K2n1h) to (K2n1E);
\draw [decorate,decoration={brace,amplitude=10pt}]
(K0.west) -- (E.west) node [midway,xshift=-3mm,yshift=0mm]
{$C_{p^2}$};
\draw [decorate,decoration={brace,amplitude=10pt}]
(K2n1.south east) -- (K0.south east)
node [midway,xshift=1mm,yshift=-1mm] {$H(n)$};
\end{tikzpicture}
\end{center}
\caption{Field diagram for
Proposition~\ref{extra special p ext}}
\end{figure}
(b) Recall the generators $g_1,\dots,g_{2n+1}$ for
$H(n)$ given in Definition~\ref{heis def}, and write
$C_{p^2}=\langle k\rangle$.  Then
\begin{align*}
\Gal(K_{2n+1}E/K_0)&\cong\{(\alpha,\beta)\in
\Gal(K_{2n+1}/K_0)\times\Gal(E/K_0):
\alpha|_{K_1}=\beta|_{K_1}\} \\
&\cong\{(g_1^{r_1}g_2^{r_2}\dots g_{2n+1}^{r_{2n+1}},k^s)
\in H(n)\times C_{p^2}:r_1\equiv s\pmod{p}\}.
\end{align*}
Let $N=\Gal(K_{2n+1}E/\Kh_{2n+1})$.  Then the
isomorphism above maps $N$ to
$\langle(g_{2n+1}^v,k^{-p})\rangle$ for some
$1\le v\le p-1$.  Let $1\le w\le p-1$ be such that
$vw\equiv1\pmod{p}$.  Then there is an onto homomorphism
$\phi:\Gal(K_{2n+1}E/K_0)\to M(n)$ defined by
$\phi(g_1,k)=h_1^v$, $\phi(g_{n+1},1)=h_{n+1}^w$, and
$\phi(g_i,1)=h_i$ for $2\le i\le n$ and
$n+2\le i\le2n+1$.  Since $\ker(\phi)=N$ this gives an
isomorphism $\Gal(\Kh_{2n+1}/K_0)\cong M(n)$.
\\[\medskipamount]
(c) For $1\le i\le 2n$ let $\gamma_i$ be the element of
$\Gal(K_{2n}/K_0)$ defined by (\ref{can gen}).  Then for
$2\le i\le n$ we have
$\gamma_i(B+C+a_{2n+1})=B+C+a_{2n+1}$.  Hence we can
extend $\gamma_i$ to $\sigmah_i\in\Aut(\Kh_{2n+1})$ by
setting $\sigmah_i(\alphah_{2n+1})=\alphah_{2n+1}$.
Let $1\le i\le n$ and let $\alphah_{2n+1}'$ be a root of
$X^p-X-\gamma_{n+i}(B+C+a_{2n+1})$.  We extend
$\gamma_{n+i}$ to $\sigmah_{n+i}\in\Aut(\Kh_{2n+1})$ by
setting $\sigmah_{n+i}(\alphah_{2n+1})=\alphah_{2n+1}'$.
By (\ref{tauB}) and (\ref{mineq1}) we have
\[\gamma_{n+i}(B+C+a_{2n+1})\equiv B+C+a_{2n+1}
\pmod{a_i}.\]
Hence by (\ref{mineq3}) and Proposition~\ref{chisum}(b)
we get
$\alphah_{2n+1}'\equiv\alphah_{2n+1}\pmod{\alpha_i}$.
It follows from Lemma~\ref{keatings lemma} that there is
a generator $\sigmah_{2n+1}$ for
$\Gal(\Kh_{2n+1}/K_{2n})$ such that
\[\sigmah_{2n+1}(\alphah_{2n+1})\equiv\alphah_{2n+1}+1
\pmod{p\alphah_{2n+1}^{p-1}}.\]
Note that statements (\ref{mshift1}), (\ref{mshift2}),
(\ref{mshift3}), (\ref{mshift4}), and (\ref{mshift5})
follow directly from these constructions.

     In order to extend $\gamma_1$ to an automorphism of
$\Kh_{2n+1}$ we let $\eta\in E$ be a root of
$X^p-X-(C+a_{2n+1})$ and let $\theta$ be a generator of
$\Gal(E/K_0)$ such that $\theta(\alpha_1)\equiv
\alpha_1+1\pmod{p\alpha_1^{p-1}}$ and $\theta(\eta)
\equiv\eta+D(1,\alpha_1)\pmod{\M_E}$.  Such $\theta$
exists by (\ref{mineq2}), (\ref{mineq5}), and
Proposition~\ref{Cp2}(c).  Let
$\sigma_1,\sigma_{2n+1}\in\Gal(K_{2n+1}/K_0)$ be defined
as in Proposition~\ref{hi he}.  Then
$\theta|_{K_1}=\sigma_1|_{K_1}$, so there is
$\rho\in\Gal(K_{2n+1}E/K_0)$ such that $\rho|_E=\theta$
and $\rho|_{K_{2n+1}}=\sigma_1$.  Define
$\sigmah_1=\rho|_{\Kh_{2n+1}}$.  Then
$\sigmah_1|_{K_{2n}}=\sigma_1|_{K_{2n}}=\gamma_1$.  By
(\ref{mineq4}), (\ref{mineq3}), and
Proposition~\ref{chisum}(b) there is a root
$\alpha_{2n+1}$ of $X^p-X-B$ such that
$\alphah_{2n+1}\equiv\alpha_{2n+1}+\eta
\pmod{p\eta^{p-1}\alpha_{2n+1}}$.  By
Corollary~\ref{refine} and (\ref{hshift3}) we get
\begin{alignat*}{2}
(\rho-1)\alphah_{2n+1}&\equiv(\rho-1)\alpha_{2n+1}
+(\rho-1)\eta&&\pmod{p\eta^{p-1}} \\
(\sigmah_1-1)\alphah_{2n+1}&\equiv
(\sigma_1-1)\alpha_{2n+1}+(\theta-1)\eta
&&\pmod{p\eta^{p-1}} \\
&\equiv D(1,\alpha_1)&&\pmod{\M_{\Kh_{2n+1}}} \\
&\equiv0&&\pmod{\alpha_1^{p-1}}.
\end{alignat*}
This proves (\ref{mshift6}).
\end{proof}

     Henceforth we will denote the $M(n)$-extensions
constructed using Proposition~\ref{extra special p ext}
by $K_{2n+1}/K_0$, rather than $\Kh_{2n+1}/K_0$.  In
addition, we will denote the generator of $K_{2n+1}$
over $K_{2n}$ by $\alpha_{2n+1}$ rather than
$\alphah_{2n+1}$, and we will denote the generators of
$\Gal(K_{2n+1}/K_0)$ by $\sigma_i$ rather than
$\sigmah_i$.

\section{Valuations of determinants}

To construct Galois scaffolds for $H(n)$- and
$M(n)$-extensions we use an approach which is similar to
that used in \cite{EK22} and \cite{EKgen}, with the
significant difference that our fields have
characteristic 0 rather than characteristic $p$.  For a
local field $F$ with residue characteristic $p$ define
$\phi:F\to F$ by $\phi(x)=x^p$; then $\phi$ can also be
applied to matrices and vectors over $F$ by acting on
the entries.  If $\ch(F)=0$ then we don't necessarily
have $\phi(x+y)=\phi(x)+\phi(y)$.  However, if
$w,x,y,z\in F$ satisfy $v_F(x)\le v_F(y)$ and
$x+y\equiv w \pmod{z}$ then we do have
$\phi(x)+\phi(y)\equiv\phi(w)\pmod{(px^p,z^p)}$.
Therefore we get the following:

\begin{lemma} \label{phidet}
Let $A=(a_{ij})\in M_k(F)$ and let $\tau$ be a
permutation of $\{1,2,\dots,k\}$ chosen so that
$\gamma=\prod_{i=1}^ka_{i,\tau(i)}$ has minimum
valuation among the $k!$ terms in the Leibniz expansion
of $\det(A)$.  Then
\begin{enumerate}[(a)]
\item
$\det(\phi(A))\equiv\phi(\det(A))\pmod{p\gamma^p}$.
\item If $d,\mu\in F$ satisfy
$\det(A)\equiv d\pmod{\mu}$ then
\[\det(\phi(A))\equiv\phi(d)\pmod{(p\gamma^p,\mu^p)}.\]
\end{enumerate}
\end{lemma}

     In the next section we will need to know the
valuations of the determinants of certain square
matrices with entries in $K_0$.  The following lemma
will allow us to compute these valuations:

\begin{lemma} \label{tval}
Let $F$ be a local field with residue characteristic
$p$, let $\beta_1,\dots,\beta_k\in F^{\times}$, and set
$r_i=-v_F(\beta_i)$.  Assume that
$r_1\le r_1\le\dots\le r_k$, and that for every
$1\le i<j\le k$ such that $r_i=r_j$, the images of
$\beta_i,\dots,\beta_j$ in $\M_F^{-r_i}/\M_F^{-r_i+1}$
are linearly independent over $\F_p$.  Set
\[M=\begin{bmatrix}
\beta_1&\beta_1^p&\cdots&\beta_1^{p^{k-1}} \\
\beta_2&\beta_2^p&\cdots&\beta_2^{p^{k-1}} \\
\vdots&\vdots&&\vdots \\
\beta_k&\beta_k^p&\cdots&\beta_k^{p^{k-1}}
\end{bmatrix}.\]
Then $v_F(\det(M))$ is equal to the minimum of the
valuations of the $k!$ terms in the Leibniz expansion of
$\det(M)$.  More precisely,
\begin{align*}
v_F(\det(M))
&=v_F(\beta_1\beta_2^p\dots\beta_k^{p^{k-1}}) \\
&=-(r_1+pr_2+\dots+p^{k-1}r_k).
\end{align*}
\end{lemma}

\begin{proof}
Let $s_1<\dots<s_{\ell}$ be the distinct elements of
$\{r_1,\dots,r_k\}$ and set
\[m_j=|\{1\le i\le k:r_i=s_j\}|.\]
For $1\le j\le\ell$ set $c_j=m_1+\dots+m_{j-1}$.  Then
$r_{c_j+1}=\dots=r_{c_j+{m_j}}=s_j$.  Set
\begin{align*}
M_j&=\begin{bmatrix}
\beta_{c_j+1}^{p^{c_j}}&\beta_{c_j+1}^{p^{c_j+1}}&\cdots
&\beta_{c_j+1}^{p^{c_j+m_j-1}} \\
\beta_{c_j+2}^{p^{c_j}}&\beta_{c_j+2}^{p^{c_j+1}}&\cdots
&\beta_{c_j+2}^{p^{c_j+m_j-1}} \\
\vdots&\vdots&&\vdots \\
\beta_{c_j+m_j}^{p^{c_j}}&\beta_{c_j+m_j}^{p^{c_j+1}}&\cdots
&\beta_{c_j+m_j}^{p^{c_j+m_j-1}}
\end{bmatrix} \\[3mm]
M'&=\begin{bmatrix}
M_1&0&\cdots&0 \\
0&M_2&\cdots&0 \\
\vdots&\vdots&\ddots&\vdots \\
0&0&\cdots&M_{\ell}
\end{bmatrix}.
\end{align*}
Thus $M'$ is obtained from $M$ by replacing all entries
outside the diagonal blocks with 0.  Since
$s_1<\dots<s_{\ell}$ we get
\begin{alignat}{2} \nonumber
\det(M)&\equiv\det(M')&&
\pmod{\beta_1\beta_2^p\dots\beta_k^{p^{k-1}}\M_F} \\
&\equiv\det(M_1)\det(M_2)\dots\det(M_{\ell})&&
\pmod{\beta_1\beta_2^p\dots\beta_k^{p^{k-1}}\M_F}.
\label{detprod}
\end{alignat}
We have
\begin{equation} \label{mooredet}
\det(M_j)=\beta_{c_j+1}^{p^{c_j}+p^{c_j+1}+\dots+p^{c_j+m_j-1}}
\begin{vmatrix}1&1&\cdots&1 \\
\mu_2&\mu_2^p&\cdots&\mu_2^{p^{m_j-1}} \\
\vdots&\vdots&&\vdots \\
\mu_{m_j}&\mu_{m_j}^p&\cdots&\mu_{m_j}^{p^{m_j-1}}
\end{vmatrix},
\end{equation}
where $\mu_h=\beta_{c_j+h}^{p^{c_j}}\beta_{c_j+1}^{-p^{c_j}}\in
\OO_F^{\times}$.  Since the images of
$\beta_{c_j+1},\beta_{c_j+2},\dots,\beta_{c_j+m_j}$ in
$\M_F^{-s_j}/\M_F^{-s_j+1}$ are linearly independent
over $\F_p$, the images of $1,\mu_2,\dots,\mu_{m_j}$ in
$\OO_F/\M_F$ are also linearly independent over $\F_p$.
It follows from the theory of the Moore determinant
\cite[Lemma 1.3.3]{Go96} that the determinant on
the right side of (\ref{mooredet}) is a unit.  Hence
\begin{align*}
v_F(\det(M_j))
&=(p^{c_j}+p^{c_j+1}+\dots+p^{c_j+m_j-1})v_F(\beta_{c_j+1}) \\
&=-(p^{c_j}+p^{c_j+1}+\dots+p^{c_j+m_j-1})s_j \\
&=-(p^{c_j}r_{c_j+1}+p^{c_j+1}r_{c_j+2}+\dots+p^{c_j+m_j-1}r_{c_j+m_j}).
\end{align*}
Combining this formula with (\ref{detprod}) gives the
desired result.
\end{proof}

\section{Field extensions with a generator} \label{gen}

In Proposition~\ref{hi he} we constructed a tower of
Galois extensions
\[K_0\subset K_1\subset\dots\subset K_{2n}\subset
K_{2n+1}\]
such that $\Gal(K_i/K_0)\cong C_p^i$ for $1\le i\le2n$
and $\Gal(K_{2n+1}/K_0)\cong H(n)$.  In
Proposition~\ref{extra special p ext} we constructed a
similar tower, but with $\Gal(K_{2n+1}/K_0)\cong M(n)$.
In this section we specialize these constructions to
produce $H(n)$-extensions and $M(n)$-extensions
$K_{2n+1}/K_0$ together with $Y\in K_{2n+1}$ such that
$v_{2n+1}(Y)\equiv-b_1\pmod{p^{2n+1}}$, where $b_1$ is
the smallest lower ramification number of
$K_{2n+1}/K_0$.  Since $p\nmid b_1$ it follows that
$K_{2n+1}=K_0(Y)$.  In the next section we will use the
generator $Y$ to construct Galois scaffolds for these
extensions.

     The extensions constructed in
Propositions~\ref{hi he} and \ref{extra special p ext}
depend on the choice of elements $a_1,\dots,a_{2n+1}$ in
$K_0$ satisfying certain conditions on the valuations
$v_0(a_i)=-u_i$.  In this section we retain these
hypotheses; in particular, in the cases where
$K_{2n+1}/K_0$ is an $H(n)$-extension we assume that the
condition $u_{2n+1}>u_n+u_{2n}$ from
Propositions~\ref{hi he}(c) holds.  We impose further
constraints on the $a_i$ by requiring that there are
$c,\omega_i\in K_0$ such that $a_i=c\omega_i^{p^{2n}}$
for $1\le i\le2n+1$.  Set $r=-v_0(c)$ and
$m_i=-v_0(\omega_i)$ for $1\le i\le2n+1$.  Then
$u_i=-v_0(a_i)=r+p^{2n}m_i$.

     For $1\le i\le2n$ let $\alpha_i\in K^{sep}$ satisfy
$\alpha_i^p-\alpha_i=a_i$.  Recall from
Propositions~\ref{hi he} and \ref{extra special p ext}
that we defined
\begin{align*}
B&=a_1\alpha_{n+1}+a_2\alpha_{n+2}+\dots+a_n\alpha_{2n}
\\
C&=-\sum_{i=1}^{p-1}p^{-1}\binom{p}{i}\alpha_1^ia_1^{p-i}.
\end{align*}
Set $E=B$ in the case where $K_{2n+1}/K_0$ is an
$H(n)$-extension, and $E=B+C$ in the case where
$K_{2n+1}/K_0$ is an $M(n)$-extension.  Let
$\alpha_{2n+1}\in K^{sep}$ satisfy
\[\alpha_{2n+1}^p-\alpha_{2n+1}=E+a_{2n+1}.\]
We define the following elements of $K_{2n+1}^{2n+1}$:
\[\Ev=\begin{bmatrix}0\\\vdots\\0\\E\end{bmatrix}\;\;\;
\omegav=\begin{bmatrix}\omega_1\\\vdots\\\omega_{2n}
\\\omega_{2n+1}\end{bmatrix}\;\;\;\alphav=\begin{bmatrix}
\alpha_1\\\vdots\\\alpha_{2n}\\\alpha_{2n+1}\end{bmatrix}.\]
Then we can write our system of Artin-Schreier equations
as $\phi(\alphav)-\alphav=c\phi^{2n}(\omegav)+\Ev$.  Now
define
\begin{equation} \label{Y}
Y=\det([\alphav,\omegav,\phi(\omegav),\dots,
\phi^{2n-1}(\omegav)]).
\end{equation}
Then $Y\in K_{2n+1}$.  By expanding the right side of
(\ref{Y}) in cofactors along the first column we get
\begin{equation} \label{Yexp}
Y=t_1\alpha_1+t_2\alpha_2+\dots+t_{2n+1}\alpha_{2n+1}
\end{equation}
with $t_i\in K_0$.

     Since $a_1,\dots,a_{2n}$ are reduced Artin-Schreier
constants, it follows from Lemma~\ref{tval} that
\begin{align*}
\hspace{-1mm}v_0(t_i)
&=v_0(\omega_1\omega_2^p\dots\omega_{i-1}^{p^{i-2}}
\omega_{i+1}^{p^{i-1}}\omega_{i+2}^{p^i}\dots
\omega_{2n+1}^{p^{2n-1}}) \\
&=-(m_1+pm_2+\dots+p^{i-2}m_{i-1}+p^{i-1}m_{i+1}
+p^im_{i+2}+\dots+p^{2n-1}m_{2n+1}).
\end{align*}
Hence for $1\le i\le2n$ we have
\begin{align*}
v_0(t_{i+1})-v_0(t_i)&=p^{i-1}m_{i+1}-p^{i-1}m_i \\
v_{2n+1}(t_{i+1})-v_{2n+1}(t_i)
&=p^{2n+i}(m_{i+1}-m_i) \\
&=p^i(u_{i+1}-u_i) \\
&=b_{i+1}-b_i,
\end{align*}
where $b_1,\dots,b_{2n+1}$ are the lower ramification
numbers of $K_{2n+1}/K_0$.  For $1\le i,j\le2n+1$ we
get
\begin{equation} \label{bdiff}
v_{2n+1}(t_j)-v_{2n+1}(t_i)=b_j-b_i.
\end{equation}
By Propositions~\ref{hi he}(c) and
\ref{extra special p ext}(c) we have
$v_0(\alpha_i)=-p^{-1}u_i$ for $1\le i\le2n+1$.  It
follows that for $1\le j\le2n$ and $i\ge0$ we have
\begin{align} \nonumber
v_0(t_{j+1}^{p^i}\alpha_{j+1})-v_0(t_j^{p^i}\alpha_j)
&=p^{i-2n-1}(b_{j+1}-b_j)-p^{-1}(u_{j+1}-u_j) \\
&=(p^{i+j-2n-1}-p^{-1})(u_{j+1}-u_j). \label{talphadiff}
\end{align}

\begin{proposition} \label{induct}
Let $K_{2n+1}/K_0$ be an extension constructed as in
Proposition~\ref{hi he} or
Proposition~\ref{extra special p ext} using
$a_1,a_2,\dots,a_{2n+1}\in K_0$.  
Assume that there are $c,\omega_i\in K_0$ such that
$a_i=c\omega_i^{p^{2n}}$ for $1\le i\le2n+1$.  Set
$u_i=-v_0(a_i)$, and assume $u_{2n}<e_0$ and
\[p^{2n}u_n+p^{2n-1}u_{2n}<b_{2n+1}.\]
If $K_{2n+1}/K_0$ is an $H(n)$-extension assume
(\ref{hineq1}), (\ref{hineq3}), and
$u_{2n+1}>u_n+u_{2n}$, while if $K_{2n+1}/K_0$ is an
$M(n)$-extension assume (\ref{mineq1})--(\ref{mineq5})
and
\[p^{2n+1}(1-p^{-1}+p^{-2})u_1<b_{2n+1}.\]
Define $Y$ as in (\ref{Y}).  Then for $0\le i\le2n$ we
have the following congruences modulo
$pt_{2n+1-i}^{p^i}\alpha_{2n+1-i}^p$:
\begin{equation} \label{phiiY}
\phi^i(Y)\equiv
\det([\alphav+\Ev+\phi(\Ev)+\dots+\phi^{i-1}(\Ev),
\phi^i(\omegav),\dots,\phi^{2n-1+i}(\omegav)]).
\end{equation}
\end{proposition}

\begin{proof}
If $K_{2n+1}/K_0$ is an $H(n)$-extension then
$v_0(E)\ge-u_n-p^{-1}u_{2n}$, while if $K_{2n+1}/K_0$ is
an $M(n)$-extension then
\[v_0(E)\ge\min\{-u_n-p^{-1}u_{2n},-(p-1+p^{-1})u_1\}.\]
In either case it follows from the hypotheses that
\begin{equation} \label{b2n1}
b_{2n+1}>-p^{2n}v_0(E).
\end{equation}

     We use induction on $i$.  The case $i=0$ follows
from (\ref{Y}).  Let $0\le i\le 2n-1$ and assume that
the claim holds for $i$.  It follows from the hypotheses
that $u_1\le\dots\le u_{2n+1}$, and hence that
$-v_0(\omega_1)\le\dots\le-v_0(\omega_{2n+1})$.
Therefore by Lemma~\ref{tval} the minimum valuation of
the $(2n+1)!$ terms in the Leibniz expansion of the
determinant (\ref{phiiY}) is equal either to
$v_0(t_j^{p^i}\alpha_j)$ for some $1\le j\le2n+1$, or to
$v_0(t_{2n+1}^{p^i}E^{p^{i-1}})$.  By (\ref{talphadiff})
we get
\begin{align*}
v_0(t_{j+1}^{p^i}\alpha_{j+1})-v_0(t_j^{p^i}\alpha_j)
&\le0 \text{ for }1\le j\le2n-i \\
v_0(t_{j+1}^{p^i}\alpha_{j+1})-v_0(t_j^{p^i}\alpha_j)
&\ge0\text{ for }2n-i\le j\le2n.
\end{align*}
Therefore the minimum of $v_0(t_j^{p^i}\alpha_j)$ is
achieved when $j=2n-i$.  By (\ref{bdiff}), (\ref{b2n1}).
and Corollary~\ref{upper vs lower}(b) we get
\begin{align*}
v_0(t_{2n+1}^{p^i}E^{p^{i-1}})
-v_0(t_{2n-i}^{p^i}\alpha_{2n-i})
&=p^{i-2n-1}(b_{2n+1}-b_{2n-i})+p^{i-1}v_0(E)
+p^{-1}u_{2n-i} \\
&>-p^{i-2n-1}b_{2n-i}+p^{-1}u_{2n-i}\ge0.
\end{align*}
Hence the minimum valuation of the $(2n+1)!$ terms in
the Leibniz expansion of (\ref{phiiY}) is equal to
$v_0(t_{2n-i}^{p^i}\alpha_{2n-i})$.  For
$1\le i\le2n-1$ it follows from (\ref{bdiff}) and the
assumption $u_{2n}<e_0$ that
\begin{align*}
v_0((pt_{2n+1-i}^{p^i}\alpha_{2n+1-i}^p)^p)
-v_0(p(t_{2n-i}^{p^i}\alpha_{2n-i})^p)\hspace{-2cm} \\
&=(p-1)e_0+p^{i-2n}(b_{2n+1-i}-b_{2n-i})-pu_{2n+1-i}
+u_{2n-i} \\
&=(p-1)e_0+(u_{2n+1-i}-u_{2n-i})-pu_{2n+1-i}+u_{2n-i} \\
&=(p-1)(e_0-u_{2n+1-i})>0.
\end{align*}
Hence by equation (\ref{Y}) and Lemma~\ref{phidet}(a)
(for the case $i=0$), and Lemma~\ref{phidet}(b) (for the
cases $1\le i\le2n-1$) we get the following congruences
modulo $pt_{2n-i}^{p^{i+1}}\alpha_{2n-i}^p$:
\begin{align*}
\phi^{i+1}(Y)&\equiv
\det([\phi(\alphav)+\phi(\Ev)+\phi^2(\Ev)+\dots
+\phi^i(\Ev),\phi^{i+1}(\omegav),\dots,
\phi^{2n+i}(\omegav)]) \\
&\equiv\det([\alphav+c\phi^{2n}(\omega)+\Ev+\phi(\Ev)+
\dots+\phi^i(\Ev),
\phi^{i+1}(\omegav),\dots,\phi^{2n+i}(\omegav)]) \\
&\equiv\det([\alphav+\Ev+\phi(\Ev)+\dots+\phi^i(\Ev),
\phi^{i+1}(\omegav),\dots,\phi^{2n+i}(\omegav)]).
\end{align*}
Note that the last congruence holds because
$i+1\le2n\le2n+i$.
\end{proof}

\begin{corollary} \label{Yval}
Let $K_{2n+1}/K_0$ be an extension which satisfies the
hypotheses of Proposition~\ref{induct}, and define $Y$
as in (\ref{Y}).  Then $v_{2n+1}(Y)=-b_1+v_{2n+1}(t_1)$
and $K_{2n+1}=K_0(Y)$.
\end{corollary}

\begin{proof}
The case $i=2n$ of Proposition~\ref{induct} gives
\[Y^{p^{2n}}\equiv
\det([\alphav+\Ev+\phi(\Ev)+\dots+\phi^{2n-1}(\Ev),
\phi^{2n}(\omegav),\dots,\phi^{4n-1}(\omegav)])
\pmod{pt_1^{p^{2n}}\alpha_1^p}.\]
It follows from (\ref{bdiff}) and (\ref{b2n1}) that
$v_0(E^{p^{2n-1}}t_{2n+1}^{p^{2n}})
>v_0(t_1^{p^{2n}}\alpha_1)$.  Hence by
(\ref{talphadiff}) we get
\[v_0(\det([\alphav+\Ev+\dots+\phi^{2n-1}(\Ev),
\phi^{2n}(\omegav),\dots,\phi^{4n-1}(\omegav)]))
=v_0(t_1^{p^{2n}}\alpha_1).\]
Since $u_1\le u_{2n}<e_0$ we have
$v_0(t_1^{p^{2n}}\alpha_1)<v_0(pt_1^{p^{2n}}\alpha_1^p)$.
Therefore it follows from the above that
\begin{align*}
v_0(Y^{p^{2n}})&=v_0(t_1^{2n}\alpha_1)
=-p^{-1}u_1+p^{2n}v_0(t_1) \\
v_{2n+1}(Y)&=-b_1+v_{2n+1}(t_1).
\end{align*}
Since $t_1\in K_0$ we get
$v_{2n+1}(Y)\equiv-b_1\pmod{p^{2n+1}}$.  It follows that
$p\nmid v_{2n+1}(Y)$, and hence that $K_{2n+1}=K_0(Y)$.
\end{proof}

\section{Galois scaffolds}

A Galois scaffold is a structure that can be attached to
certain totally ramified Galois $p$-extensions.  A
Galois scaffold for $L/K$ can be used to answer
questions about the Galois module structure of $\OO_L$
and its ideals which are in general difficult to
address.  In \cite{BE18}, Byott and Elder gave
sufficient conditions for a totally ramified Galois
$p$-extension $L/K$ to admit a Galois scaffold.  In this
section we define Galois scaffolds and state the
Byott-Elder criterion for the existence of a Galois
scaffold.  In Section~\ref{appex} we will use Galois
scaffolds to get information about Galois module theory.

     For $n\geq 1$ set
$\SS_{p^n}=\{0,1,\dots,p^n-1\}$.  Write the
base-$p$ expansion of $s\in\SS_{p^n}$ as
$s=\sum_{i=0}^{n-1}s_{(i)}p^{i}$, with $0\le s_{(i)}<p$.
Let $K$ be a local field with residue characteristic
$p$, let $L/K$ be a totally ramified Galois extension
of degree $p^n$, and set $G=\Gal(L/K)$.  Assume that the
lower ramification numbers $b_1\leq \hdots \leq b_n$ of
$L/K$ are relatively prime to $p$.  Associated to the
extension $L/K$ there is a function $\bb:\SS_{p^n}\to\Z$
defined by
\[\bb(s)=\sum_{i=1}^n s_{(n-i)}p^{n-i}b_i.\]

     Let $r:\Z\to \SS_{p^n}$ be the residue function,
defined by $r(a)\in\SS_{p^n}$ and
$r(a)\equiv a\pmod{p^n}$. The assumption $p\nmid b_i$ for
$1\leq i\leq n$ implies that
$r\circ(-\bb):\SS_{p^n}\to\SS_{p^n}$ is a bijection.
Let $\aa:\SS_{p^n}\to\SS_{p^n}$ be the inverse of
$r\circ(-\bb)$.  We may extend $\aa$ to a function from
$\Z$ to $\SS_{p^n}$ by setting $\aa(t)=\aa(r(t))$ for
$t\in\Z$.

\begin{definition}
Let $L/K$, $b_1,\dots,b_n$, $\bb$ and $\aa$ be as above,
and let $\mathfrak{c}\geq1$.  Then a Galois scaffold for
$L/K$ with precision $\mathfrak{c}$ consists of:
\begin{enumerate}[(i)]
\item Elements $\lambda_t\in L$ for each $t\in\Z$ such
that $v_L(\lambda_t)=t$ and
$\lambda_{t}\lambda_s^{-1}\in K $ whenever
$t\equiv s\pmod{p^n}$.
\item Elements $\Psi_1,\dots,\Psi_n$ in the augmentation
ideal $(\sigma-1:\sigma\in G)$ of $K[G]$ such that for
each $1\leq i\leq n$ and $t\in\Z$,
\begin{alignat*}{3}
\Psi_i\lambda_t&\equiv
u_{i,t}\lambda_{t+p^{n-i}b_i}
&&\pmod{\lambda_{t+p^{n-i}b_i}\M_L^{\mathfrak{c}}}
&&\quad\text{ if }\aa(t)_{(n-i)}\geq 1 \\
\Psi_i\lambda_t&\equiv
0&&\pmod{\lambda_{t+p^{n-i}b_i}\M_L^{\mathfrak{c}}}
&&\quad\text{ if }\aa(t)_{(n-i)}=0
\end{alignat*}
for some $u_{i,t}\in\OO_K^\times$.
\end{enumerate}
\end{definition}

     We now describe the sufficient conditions for the
existence of a Galois scaffold given in \cite{BE18}.
Let $K_n/K_0$ be a totally ramified Galois extension of
degree $p^n$.  Set $G=\Gal(K_n/K_0)$ and let
\[\{1\}=G^{(n)}\le G^{(n-1)}\le\dots\le G^{(1)}\le
G^{(0)}=G\]
be a refinement of the ramification filtration of $G$ by
normal subgroups of $G$ such that $|G^{(i)}|=p^{n-i}$.
For $1\le i\le n$ let $K_i$ be the fixed field of
$G^{(i)}$ and choose
$\sigma_i\in G^{(i-1)}\smallsetminus G^{(i)}$.  Then
$\sigma_i|_{K_i}$ is a generator for
$\Gal(K_i/K_{i-1})\cong G^{(i-1)}/G^{(i)}\cong C_p$.

     Let $b_1\le b_2\le\dots\le b_n$ and
$u_1\le u_2\le\dots\le u_n$ be the lower and upper
ramification numbers of $K_n/K_0$, counted with
multiplicities.  Using Proposition~3 of
\cite[IV\,\S1]{cl} we find that the lower ramification
numbers of $K_i/K_0$ are $b_1\le\dots\le b_i$ and the
upper ramification numbers are $u_1\le\dots\le u_i$.
Assume that the lower ramification numbers satisfy
$p\nmid b_1$ and $b_i\equiv b_1\pmod{p^n}$ for
$1\le i\le n$.  Let $X_j\in K_j$ satisfy
$v_j(X_j)=-b_j$.  Then for $1\le i\le j$ we have
$v_j((\sigma_i-1)X_j)=b_i-b_j$.  It follows that
$(\sigma_i-1)X_j=\mu_{ij}+\epsilon_{ij}$ for some
$\mu_{ij}\in K_0$ and $\epsilon_{ij}\in K_j$ such that
$v_j(\mu_{ij})=b_i-b_j$ and
$v_j(\epsilon_{ij})>b_i-b_j$.

\begin{theorem} \label{scafcond}
Let $K_n/K_0$, $u_i$, $b_i$, $\sigma_i$, $X_i$,
$\mu_{ij}$, and $\epsilon_{ij}$ be as above.  Set
\[\cc=\min_{1\le i\le j\le n}
\{v_n(\epsilon_{ij})-v_n(\mu_{ij})-p^{n-1}u_i+p^{n-j}b_i\}.\]
If $\cc\ge1$ then $K_n/K_o$ admits a Galois scaffold
with precision $\cc$.
\end{theorem}

\begin{proof}
This is proved as Theorem~2.10 in \cite{BE18}.
\end{proof}

\section{Extensions with a Galois scaffold} \label{build}

In Section~\ref{gen} we constructed totally ramified
$H(n)$-extensions and $M(n)$-extensions $K_{2n+1}/K_0$
and elements $Y\in K_{2n+1}$ such that
$K_{2n+1}=K_0(Y)$.  In this section we specialize these
constructions to produce $H(n)$-extensions and
$M(n)$-extensions which satisfy the conditions of
Theorem~\ref{scafcond}, and therefore have Galois
scaffolds.

     To apply Theorem~\ref{scafcond} we need to find
elements $X_j\in K_j$ which have the properties
specified in the theorem.  Since $K_{2n}/K_0$ is an
elementary abelian $p$-extension we can use known
results to get $X_j$ for $1\le j\le2n$:

\begin{proposition} \label{elemX}
Let $K_{2n+1}/K_0$ be an extension constructed using
Proposition~\ref{induct}.  Then for
$1\le j\le2n$ there are $X_j\in K_j$ such that
$v_j(X_j)=-b_j$ with the following property: For
$1\le i\le j$ we have
$(\sigma_i-1)X_j=\mu_{ij}+\epsilon_{ij}$ with
$\mu_{ij}\in K_0$, $\epsilon_{ij}\in K_j$,
$v_j(\mu_{ij})=b_i-b_j$, and
\[v_{2n+1}(\epsilon_{ij})-v_{2n+1}(\mu_{ij})
=p^{2n+1}e_0-(p-1)p^{2n}u_i.\]
Furthermore, we have $\mu_{jj}=1$.
\end{proposition}

\begin{proof}
See the proof of Lemma~3.8 in \cite{BE18}.
\end{proof}

     It follows from Propositions~\ref{elemX} and
\ref{decreasing} that
\begin{align} \nonumber
\min_{1\le i\le j\le2n}\{v_{2n+1}(\epsilon_{ij})
-v_{2n+1}(\mu_{ij})-p^{2n}u_i+p^{2n+1-j}b_i\}\hspace{-2cm}
\\
&=\min_{1\le i\le j\le2n}\{p^{2n+1}e_0-p^{2n+1}u_i
+p^{2n+1-j}b_i\} \nonumber \\
&=\min_{1\le i\le2n}\{p^{2n+1}e_0-p^{2n+1}u_i+pb_i\}
\nonumber \\
&=p^{2n+1}e_0-p^{2n+1}u_{2n}+pb_{2n}.
\label{elemc}
\end{align}
Suppose $p^{2n+1}e_0-p^{2n+1}u_{2n}+pb_{2n}>0$.  In
order to apply Theorem~\ref{scafcond} to get a Galois
scaffold for $K_{2n+1}/K_0$ we need to choose an
appropriate $X_{2n+1}\in K_{2n+1}$.  The remainder of
this section is devoted to showing that
$X_{2n+1}=t_{2n+1}^{-1}Y$ satisfies the conditions of
Theorem~\ref{scafcond} and can be used to construct
$H(n)$-extensions and $M(n)$-extensions with Galois
scaffolds under suitable assumptions on ramification
data.

\begin{theorem} \label{Hnscaf}
Let $K_0$ be a local field with residue characteristic
$p>2$.  Let $c,\omega_1,\dots,\omega_{2n+1}$ be elements
of $K_0$ and set $a_i=c\omega_i^{p^{2n}}$ for
$1\le i\le2n+1$.  For $1\le i\le2n$ let $\alpha_i$ satisfy
$\alpha_i^p-\alpha_i=a_i$ and define $K_1,\dots,K_{2n}$
by $K_i=K_{i-1}(\alpha_i)$.  Set
\[B=a_1\alpha_{n+1}+a_2\alpha_{n+2}+\dots
+a_n\alpha_{2n},\]
let $\alpha_{2n+1}$ satisfy
$\alpha_{2n+1}^p-\alpha_{2n+1}=B+a_{2n+1}$, and define
$K_{2n+1}=K_{2n}(\alpha_{2n+1})$.  Set $u_i=-v_0(a_i)$
for $1\le i\le2n+1$ and define $b_1,\dots,b_{2n+1}$
recursively by $b_1=u_1$ and
$b_{i+1}-b_i=p^i(u_{i+1}-u_i)$ for $1\le i\le2n$.
Assume that $a_1,\dots,a_{2n}$ are reduced
Artin-Schreier constants, and that
\begin{align}
u_n+(1-p^{-1})u_{2n}&<e_0 \label{Hineq1} \\
p^{-1}u_n+p^{-2}u_{2n}+(1-p^{-1})u_{2n+1}&<e_0
\label{Hineq2} \\
u_{2n}&<e_0 \label{Hineq3} \\
u_{2n+1}-p^{-2n-1}b_{2n+1}&<e_0 \label{Hineq5} \\
p^{2n}u_n+p^{2n}u_{2n}&<b_{2n+1}. \label{Hineq4}
\end{align}
Then $K_{2n+1}/K_0$ is an $H(n)$-extension with upper
ramification numbers $u_1\le\dots\le u_{2n+1}$ and lower
ramification numbers $b_1\le\dots\le b_{2n+1}$.
Furthermore, $K_{2n+1}/K_0$ has a Galois scaffold with
precision
\[\cc=\min\{b_{2n+1}-p^{2n}u_n-p^{2n}u_{2n},
p^{2n+1}e_0-p^{2n+1}u_{2n+1}+b_{2n+1}\}.\]
\end{theorem}

\begin{proof}
By (\ref{Hineq4}) and Corollary~\ref{upper vs lower}(b)
we get $u_n+u_{2n}<u_{2n+1}$.  Therefore by
Proposition~\ref{hi he}(c), $K_{2n+1}/K_0$ is an
$H(n)$-extension with upper ramification numbers
$u_1\le\dots\le u_{2n+1}$ and lower ramification numbers
$b_1\le\dots\le b_{2n+1}$.  By (\ref{Hineq3}),
(\ref{Hineq4}), and Corollary~\ref{Yval} we get
$v_{2n+1}(Y)=-b_1+v_{2n+1}(t_1)$ and $K_{2n+1}=K_0(Y)$.
Set $X_{2n+1}=t_{2n+1}^{-1}Y$.  Then by (\ref{bdiff}) we
get
\[v_{2n+1}(X_{2n+1})=-v_{2n+1}(t_{2n+1})-b_1+v_{2n+1}(t_1)
=-b_{2n+1}.\]
Let $1\le i\le n$.  Then by (\ref{Yexp}),
(\ref{hshift1}), and (\ref{hshift3}) we have
\begin{align*}
(\sigma_i-1)(X_{2n+1})
&=t_{2n+1}^{-1}t_i(\sigma_i-1)(\alpha_i) \\
(\sigma_{n+i}-1)(X_{2n+1})
&=t_{2n+1}^{-1}t_{n+i}(\sigma_{n+i}-1)(\alpha_{n+i})
+(\sigma_{n+i}-1)(\alpha_{2n+1}) \\
(\sigma_{2n+1}-1)(X_{2n+1})
&=(\sigma_{n+i}-1)(\alpha_{2n+1}).
\end{align*}

     For $1\le i\le2n+1$ set
$\mu_{i,2n+1}=t_{2n+1}^{-1}t_i\in K_0$.  Then
$v_{2n+1}(\mu_{i,2n+1})=b_i-b_{2n+1}$ by (\ref{bdiff}).
Write $(\sigma_i-1)(X_{2n+1})
=\mu_{i,2n+1}+\epsilon_{i,2n+1}$ with
$\epsilon_{i,2n+1}\in K_{2n+1}$.  Then by
(\ref{hshift2}) and (\ref{hshift4}) we have
\begin{align*}
v_{2n+1}(\epsilon_{i,2n+1})-v_{2n+1}(\mu_{i,2n+1})
&\ge p^{2n+1}e_0-(p-1)p^{2n}u_i \\
v_{2n+1}(\epsilon_{n+i,2n+1})-v_{2n+1}(\mu_{n+i,2n+1})
&\ge\min\{p^{2n+1}e_0-(p-1)p^{2n}u_{n+i},
b_{2n+1}-b_{n+i}-p^{2n}u_i\} \\
v_{2n+1}(\epsilon_{2n+1,2n+1})-v_{2n+1}(\mu_{2n+1,2n+1})
&\ge p^{2n+1}e_0-(p-1)p^{2n}u_{2n+1}
\end{align*}
for $1\le i\le n$.  It follows that
\begin{align*}
v_{2n+1}(\epsilon_{i,2n+1})-v_{2n+1}(\mu_{i,2n+1})
-p^{2n}u_i+b_i&\ge p^{2n+1}e_0-p^{2n+1}u_i+b_i \\
v_{2n+1}(\epsilon_{n+i,2n+1})-v_{2n+1}(\mu_{n+i,2n+1})
-p^{2n}u_{n+i}+b_{n+i}
&\ge\min\{p^{2n+1}e_0-p^{2n+1}u_{n+i}+b_{n+i}, \\
&\hspace{1.4cm}b_{2n+1}-p^{2n}u_i-p^{2n}u_{n+i}\} \\
v_{2n+1}(\epsilon_{2n+1,2n+1})-v_{2n+1}(\mu_{2n+1,2n+1})
-p^{2n}u_{2n+1}+b_{2n+1}
&\ge p^{2n+1}e_0-p^{2n+1}u_{2n+1}+b_{2n+1}.
\end{align*}
Using Corollary~\ref{upper vs lower}(a) we get
\begin{equation} \label{imin}
p^{2n+1}e_0-p^{2n+1}u_i+b_i
\ge p^{2n+1}e_0-p^{2n+1}u_{2n+1}+b_{2n+1}
\end{equation}
for $1\le i\le2n$.  It follows that
\begin{align*}
\min_{1\le i\le2n+1}\{v_{2n+1}(\epsilon_{i,2n+1})
-v_{2n+1}(\mu_{i,2n+1})-p^{2n}u_i+b_i\}\ge\hspace{-5cm} \\
&\min\{b_{2n+1}-p^{2n}u_n-p^{2n}u_{2n},
p^{2n+1}e_0-p^{2n+1}u_{2n+1}+b_{2n+1}\}.
\end{align*}
Using the relation
$b_{2n+1}-b_{2n}=p^{2n}(u_{2n+1}-u_{2n})$ we get
\begin{equation} \label{rel}
p^{2n+1}e_0-p^{2n+1}u_{2n}+pb_{2n}
>p^{2n+1}e_0-p^{2n+1}u_{2n+1}+b_{2n+1}.
\end{equation}
Hence by (\ref{elemc}) and Theorem~\ref{scafcond} we
deduce that $K_{2n+1}/K_0$ has a Galois scaffold with
precision
\[\cc=\min\{b_{2n+1}-p^{2n}u_n-p^{2n}u_{2n},\;
p^{2n+1}e_0-p^{2n+1}u_{2n+1}+b_{2n+1}\}.\]
Note that $\cc\ge1$ by (\ref{Hineq4}) and
(\ref{Hineq5}).
\end{proof}

     By replacing (\ref{Hineq1})--(\ref{Hineq5}) with
the single inequality $u_{2n+1}\le e_0$ we get the
following simpler, but weaker, version of
Theorem~\ref{Hnscaf}:

\begin{corollary} \label{Hnsimp}
Let $K_{2n+1}/K_0$ be an extension constructed as in
Theorem~\ref{Hnscaf}, but with
(\ref{Hineq1})--(\ref{Hineq5}) replaced with the single
assumption $u_{2n+1}\le e_0$.  Then $K_{2n+1}/K_0$ is an
$H(n)$-extension which has a Galois scaffold with
precision $\cc=b_{2n+1}-p^{2n}u_n-p^{2n}u_{2n}$.
\end{corollary}

\begin{proof}
As in the proof of Theorem~\ref{Hnscaf}, (\ref{Hineq4})
implies $u_n+u_{2n}<u_{2n+1}$.   Combining this with the
assumption $u_{2n+1}\le e_0$ we find that the
assumptions (\ref{Hineq1})--(\ref{Hineq5}) from
Theorem~\ref{Hnscaf} all hold.  Since $u_{2n+1}\le e_0$
we have
\begin{equation} \label{bound}
b_{2n+1}-p^{2n}u_n-p^{2n}u_{2n}<
p^{2n+1}e_0-p^{2n+1}u_{2n+1}+b_{2n+1}.
\end{equation}
Hence by Theorem~\ref{Hnscaf} $K_{2n+1}/K_0$ is an
$H(n)$-extension which admits a Galois scaffold with
precision $\cc=b_{2n+1}-p^{2n}u_n-p^{2n}u_{2n}$.
\end{proof}

\begin{theorem} \label{Mnscaf}
Let $K_0$ be a local field with residue characteristic
$p>2$.  Let $c,\omega_1,\dots,\omega_{2n+1}$ be elements
of $K_0$ and set $a_i=c\omega_i^{p^{2n}}$ for
$1\le i\le2n+1$.  For $1\le i\le2n$ let $\alpha_i$ satisfy
$\alpha_i^p-\alpha_i=a_i$ and define $K_1,\dots,K_{2n}$
by $K_i=K_{i-1}(\alpha_i)$.  Set
\begin{align*}
B&=a_1\alpha_{n+1}+a_2\alpha_{n+2}+\dots
+a_n\alpha_{2n}, \\
C&=D(\alpha_1,a_1)
=-\sum_{i=1}^{p-1}p^{-1}\binom{p}{i}\alpha_1^ia_1^{p-i},
\end{align*}
let $\alpha_{2n+1}$ satisfy
$\alpha_{2n+1}^p-\alpha_{2n+1}=B+C+a_{2n+1}$, and define
$K_{2n+1}=K_{2n}(\alpha_{2n+1})$.  Set $u_i=-v_0(a_i)$
for $1\le i\le2n+1$ and define $b_1,\dots,b_{2n+1}$
recursively by $b_1=u_1$ and
$b_{i+1}-b_i=p^i(u_{i+1}-u_i)$ for $1\le i\le2n$.
Assume that $a_1,\dots,a_{2n}$ are reduced
Artin-Schreier constants, and that
\begin{align}
u_n+(1-p^{-1})u_{2n}&<e_0 \label{Mineq1} \\
p^{-1}u_n+p^{-2}u_{2n}+(1-p^{-1})u_{2n+1}&<e_0
\label{Mineq2} \\
(1-p^{-1}+p^{-2})u_1+(1-p^{-1})u_{2n+1}&<e_0
\label{Mineq6} \\
u_{2n}&<e_0 \label{Mineq3} \\
u_{2n+1}-p^{-2n-1}b_{2n+1}&<e_0 \label{Mineq5} \\
p^{2n}u_n+p^{2n}u_{2n}&<b_{2n+1} \label{Mineq4} \\
p^{2n+1}u_1&<b_{2n+1}. \label{Mineq7}
\end{align}
Then $K_{2n+1}/K_0$ is an $M(n)$-extension with upper
ramification numbers $u_1\le\dots\le u_{2n+1}$ and lower
ramification numbers $b_1\le\dots\le b_{2n+1}$.
Furthermore, $K_{2n+1}/K_0$ has a Galois scaffold with
precision
\[\cc=\min\{b_{2n+1}-p^{2n+1}u_1,\;
b_{2n+1}-p^{2n}u_n-p^{2n}u_{2n},\;
p^{2n+1}e_0-p^{2n+1}u_{2n+1}+b_{2n+1}\}.\]
\end{theorem}

\begin{proof}
As in the proof of Theorem~\ref{Hnscaf}, (\ref{Mineq4})
implies $u_n+u_{2n}<u_{2n+1}$.  In addition, by
(\ref{Mineq7}) and Corollary~\ref{upper vs lower}(a) we get
$(p+1-p^{-1})u_1<u_{2n+1}$.  Hence by
Proposition~\ref{extra special p ext}, $K_{2n+1}/K_0$ is
an $M(n)$-extension with upper ramification numbers
$u_1\le\dots\le u_{2n+1}$ and lower ramification
numbers $b_1\le\dots\le b_{2n+1}$.  By (\ref{Mineq3}),
(\ref{Mineq4}), (\ref{Mineq7}), and Corollary~\ref{Yval}
we get $v_{2n+1}(Y)=-b_1+v_{2n+1}(t_1)$ and
$K_{2n+1}=K_0(Y)$.  Set $X_{2n+1}=t_{2n+1}^{-1}Y$; then
$v_{2n+1}(X_{2n+1})=-b_{2n+1}$.  By (\ref{Yexp}),
(\ref{mshift1}), and (\ref{mshift3}) we get
\begin{alignat*}{2}
(\sigma_1-1)(X_{2n+1})
&=t_{2n+1}^{-1}t_1(\sigma_1-1)(\alpha_1)
+(\sigma_1-1)(\alpha_{2n+1}), \\
(\sigma_i-1)(X_{2n+1})
&=t_{2n+1}^{-1}t_i(\sigma_i-1)(\alpha_i)
&&\quad(2\le i\le n), \\
(\sigma_{n+i}-1)(X_{2n+1})
&=t_{2n+1}^{-1}t_{n+i}(\sigma_{n+i}-1)(\alpha_{n+i})
+(\sigma_{n+i}-1)(\alpha_{2n+1})&&\quad(1\le i\le n), \\
(\sigma_{2n+1}-1)(X_{2n+1})
&=(\sigma_{n+i}-1)(\alpha_{2n+1}).
\end{alignat*}

     For $1\le i\le2n+1$ set
$\mu_{i,2n+1}=t_{2n+1}^{-1}t_i$; then
$v_{2n+1}(\mu_{i,2n+1})=b_i-b_{2n+1}$.  Write
$(\sigma_i-1)(X_{2n+1})=\mu_{i,2n+1}+\epsilon_{i,2n+1}$
with $\epsilon_{i,2n+1}\in K_{2n+1}$.  Then by
(\ref{mshift2}), (\ref{mshift6}), (\ref{mshift4}), and
(\ref{mshift5}) we have
\begin{alignat*}{2}
v_{2n+1}(\epsilon_{1,2n+1})-v_{2n+1}(\mu_{1,2n+1})
&\ge\min\{p^{2n+1}e_0-(p-1)p^{2n}u_1, \\
&\hspace*{4cm}b_{2n+1}-b_1-(p-1)p^{2n}u_1\},\hspace{-2cm}
\\[1mm]
v_{2n+1}(\epsilon_{i,2n+1})-v_{2n+1}(\mu_{i,2n+1})
&\ge p^{2n+1}e_0-(p-1)p^{2n}u_i
&&(2\le i\le n), \\[1mm]
v_{2n+1}(\epsilon_{n+i,2n+1})-v_{2n+1}(\mu_{n+i,2n+1})
&\ge\min\{p^{2n+1}e_0-(p-1)p^{2n}u_{n+i}, \\
&\hspace{2.5cm}b_{2n+1}-b_{n+i}-p^{2n}u_i\}&&(1\le i\le n),
\\[1mm]
v_{2n+1}(\epsilon_{2n+1,2n+1})-v_{2n+1}(\mu_{2n+1,2n+1})
&\ge p^{2n+1}e_0-(p-1)p^{2n}u_{2n+1}.
\end{alignat*}
It follows that
\begin{align*}
v_{2n+1}(\epsilon_{1,2n+1})-v_{2n+1}(\mu_{1,2n+1})
-p^{2n}u_1+b_1
&\ge\min\{p^{2n+1}e_0-p^{2n+1}u_1+b_1, \\
&\hspace*{2.5cm}b_{2n+1}-p^{2n+1}u_1\}, \\
v_{2n+1}(\epsilon_{2n+1,2n+1})-v_{2n+1}(\mu_{2n+1,2n+1})
-p^{2n}u_{2n+1}+b_{2n+1}
&\ge p^{2n+1}e_0-p^{2n+1}u_{2n+1}+b_{2n+1}.
\end{align*}
In addition, for $2\le i\le n$ we have
\[v_{2n+1}(\epsilon_{i,2n+1})-v_{2n+1}(\mu_{i,2n+1})
-p^{2n}u_i+b_i\ge p^{2n+1}e_0-p^{2n+1}u_i+b_i\]
and for $1\le i\le n$ we have
\begin{align*}
v_{2n+1}(\epsilon_{n+i,2n+1})-v_{2n+1}(\mu_{n+i,2n+1})
-p^{2n}u_{n+i}+b_{n+i}
&\ge\min\{p^{2n+1}e_0-p^{2n+1}u_{n+i}+b_{n+i}, \\
&\hspace{1.7cm}b_{2n+1}-p^{2n}u_i-p^{2n}u_{n+i}\}.
\end{align*}
Clearly $b_{2n+1}-p^{2n}u_i-p^{2n}u_{n+i}$ is minimized
for $1\le i\le n$ by taking $i=n$.  It now follows from
(\ref{imin}), (\ref{rel}), (\ref{elemc}), and
Theorem~\ref{scafcond} that $K_{2n+1}/K_0$ has a Galois
scaffold with the precision $\cc$ given in the statement
of the theorem.
\end{proof}

     Theorem~\ref{Mnscaf}, like Theorem~\ref{Hnscaf},
can be simplified by strengthening the hypotheses:

\begin{corollary} \label{Mnsimp}
Let $K_{2n+1}/K_0$ be an extension constructed as in
Theorem~\ref{Mnscaf}, but with
(\ref{Mineq1})--(\ref{Mineq5}) replaced with the single
assumption $u_{2n+1}\le e_0$.  Then $K_{2n+1}/K_0$ is an
$M(n)$-extension which has a Galois scaffold with
precision
\[\cc=\min\{b_{2n+1}-p^{2n+1}u_1,\;
b_{2n+1}-p^{2n}u_n-p^{2n}u_{2n}\}.\]
\end{corollary}

\begin{proof}
As in the proof of Theorem~\ref{Hnscaf}, (\ref{Mineq4})
implies $u_n+u_{2n}<u_{2n+1}$.  In addition, using
Corollary~\ref{upper vs lower} and (\ref{Mineq7}) we get
$pu_1<u_{2n+1}$.  Combining these inequalities with
$u_{2n+1}\le e_0$ gives assumptions
(\ref{Mineq1})--(\ref{Mineq5}) of Theorem~\ref{Mnscaf}.
Since $u_{2n+1}\le e_0$, the inequality (\ref{bound}) is
valid.  Therefore by Theorem~\ref{Mnscaf} $K_{2n+1}/K_0$
is an $M(n)$-extension which admits a Galois scaffold
with the given precision.
\end{proof}

\begin{remark}
Suppose $\ch(K)=p$.  In this setting,
Theorem~\ref{Hnscaf} reduces to Corollary~\ref{Hnsimp},
and Theorem~\ref{Mnscaf} reduces to
Corollary~\ref{Mnsimp}.  In \cite{EKgen},
$G$-extensions $L/K$ with a Galois scaffold were
constructed for an arbitrary finite $p$-group $G$.  In
the cases where $G$ is an extraspecial $p$-group, the
approach taken in this paper gives a method for
constructing $G$-extensions with a Galois scaffold which
is more explicit than the method used in \cite{EKgen}.
\end{remark}

\section{Applications and examples} \label{appex}

The results of the previous section can be used to get
information about the Galois module structure of rings
of integers.  Let $L/K$ be a finite Galois extension of
local fields and set $G=\Gal(L/K)$.  Recall that the
associated order of the ring of integers $\OO_L$ of $L$
is defined to be
\[\A_{L/K}=\{\lambda\in K[G]:
\lambda(\OO_L)\subset\OO_L\}.\]

\begin{theorem} \label{GMS}
Let $G$ be equal to either $H(n)$ or $M(n)$ for some
$n\ge1$.  Let $K_{2n+1}/K_0$ be a totally ramified
$G$-extension constructed using Theorem~\ref{Hnscaf} or
Theorem~\ref{Mnscaf}, so that $K_{2n+1}/K_0$ has a
Galois scaffold with precision $\cc$ for some $\cc\ge1$.
Let $u_1\le\dots\le u_{2n+1}$ be the upper ramification
numbers and $b_1\le\dots\le b_{2n+1}$ the lower
ramification numbers of $K_{2n+1}/K_0$.  Let $r(u_1)$
denote the least nonnegative residue of $u_1$ modulo
$p^{2n+1}$.
\begin{enumerate}[(a)]
\item If $\cc\ge r(u_1)$ and $r(u_1)\mid p^m-1$ for some
$1\le m\le2n+1$ then $\OO_{2n+1}$ is free over its
associated order $\A_{K_{2n+1}/K_0}$.
\item If $\cc\ge2p^{2n+1}-1$ and $r(u_1)=p^{2n+1}-1$
then $\OO_{2n+1}$ is free over $\A_{K_{2n+1}/K_0}$ and
$\A_{K_{2n+1}/K_0}$ is a Hopf order in $K_0[G]$.
Furthermore, $\A_{K_{2n+1}/K_0}$ is a noncommutative
local ring whose only idempotents are 0 and 1.
\end{enumerate}
\end{theorem}

\begin{proof}
(a) Since $b_{2n+1}\equiv b_1\pmod{p^{2n+1}}$ and $b_1=u_1$
we have $r(b_{2n+1})=r(u_1)$.  Therefore the claim
follows from Theorem~4.8 of \cite{BCE18}.
\\[\medskipamount]
(b) It follows from (a) that $\OO_{2n+1}$ is free over
$\A_{K_{2n+1}/K_0}$.  Using the bound
$\cc\ge p^{2n+1}+r(u_1)$, it is shown in the proof of
Theorem~3.6 of \cite{BCE18} (pp.~985--6) that
$\A_{K_{2n+1}/K_0}$ has a unique maximal ideal $\M$,
that $\A_{K_{2n+1}/K_0}/\M\cong\OO_0/\M_0$, and that
there is $d\ge1$ such that
$\M^d\subset\pi_0\A_{K_{2n+1}/K_0}$.  It follows that
every element of $\A_{K_{2n+1}/K_0}\smallsetminus\M$ is
a unit.  Let $e\in\A_{K_{2n+1}/K_0}$ be idempotent, with
$e\not=1$.  Then $e\not\in\A_{K_{2n+1}/K_0}^{\times}$,
so $e\in\M$.  Hence for all $t\ge1$ we have
$e=e^{td}\in\pi_0^t\A_{K_{2n+1}/K_0}$.  Since
$\A_{K_{2n+1}/K_0}$ is a free $\OO_0$-module, this
implies $e=0$.  Therefore the only idempotents of
$\A_{K_{2n+1}/K_0}$ are 0 and 1, so $\A_{K_{2n+1}/K_0}$
is indecomposable as a left $\A_{K_{2n+1}/K_0}$-module.
Since $\A_{K_{2n+1}/K_0}$ is a free $\OO_0$-module and
$\A_{K_{2n+1}/K_0}/\OO_0[G]$ is a torsion
$\OO_0$-module, it follows that $\A_{K_{2n+1}/K_0}$ is
indecomposable as a left $\OO_0[G]$-module.  In
addition, since $b_i\equiv-1\pmod{p^{2n+1}}$ for
$1\le i\le 2n+1$, the different of $K_{2n+1}/K_0$ is
generated by an element of $K$.  It now follows from
Proposition~4.5.2 of \cite{Bon02} that
$\A_{K_{2n+1}/K_0}$ is a Hopf order in $K_0[G]$.
\end{proof}

\begin{example} \label{ex}
Let $n\ge1$ and let $K_0$ be a local field of
characteristic 0 whose residue field $\OO_0/\M_0$ has
characteristic $p>2$ and cardinality $\ge p^{2n}$.  Let
$u\ge1$ with $p\nmid u$ and let $c\in K_0$ satisfy
$v_0(c)=-u$.  Let $\omega_1,\dots,\omega_{2n}$ be
elements of $\OO_0$ whose images in $\OO_0/\M_0$ are
linearly independent over $\F_p$.  For $1\le i\le2n$ set
$a_i=c\omega_i^{p^{2n}}$ and let $\alpha_i$ be a root of
$X^p-X-a_i$.  Define $K_1,\dots,K_{2n}$ recursively by
$K_i=K_{i-1}(\alpha_i)$ for $1\le i\le2n$.  Let $t\ge1$
and let $\omega_{2n+1}\in K_0$ satisfy
$v_0(\omega_{2n+1})=-t$.  Set
$a_{2n+1}=c\omega_{2n+1}^{p^{2n}}$.  Then in the
notation of Theorems~\ref{Hnscaf} and \ref{Mnscaf} we
have $u_i=b_i=u$ for $1\le i\le2n$,
$u_{2n+1}=tp^{2n}+u$, and $b_{2n+1}=tp^{4n}+u$.

     Now assume that $tp^{2n}+u\le e_0$ and
$2u\le tp^{2n}$.  Define $B$ as in Theorem~\ref{Hnscaf},
let $\alpha_{2n+1}$ be a root of $X^p-X-(B+a_{2n+1})$,
and set $K_{2n+1}=K_{2n}(\alpha_{2n+1})$.  Then by
Corollary~\ref{Hnsimp}, $K_{2n+1}/K_0$ is an
$H(n)$-extension which has a Galois scaffold with
precision $\cc=tp^{4n}+u-2up^{2n}$.  Hence by
Theorem~\ref{GMS}(a) if $r(u)\mid p^m-1$ for some $m$
such that $1\le m\le 2n+1$ then $\OO_{2n+1}$ is free
over its associated order $\A_{K_{2n+1}/K_0}$.  In
addition, if $r(u)=p^{2n+1}-1$ and $2u+2p\le tp^{2n}$
then $\A_{K_{2n+1}/K_0}$ is a Hopf order by
Theorem~\ref{GMS}(b).  For instance, if $u=t=1$ and
$e_0\ge p^{2n}+1$ then $\OO_{2n+1}$ is free over
$\A_{K_{2n+1}/K_0}$, while if $u=p^{2n+1}-1$, $t=2p+1$,
and $e_0\ge3p^{2n+1}+p^{2n}-1$ then $\A_{K_{2n+1}/K_0}$
is a Hopf order.

     Alternatively, we can assume that $tp^{2n}+u\le e_0$
and $pu\le tp^{2n}$, and define $B$ and $C$ as in
Theorem~\ref{Mnscaf}.  We let $\alpha_{2n+1}$ be a root of
$X^p-X-(B+C+a_{2n+1})$ and set
$K_{2n+1}=K_{2n}(\alpha_{2n+1})$.  Then by
Corollary~\ref{Mnsimp}, $K_{2n+1}/K_0$ is an
$M(n)$-extension which has a Galois scaffold with
precision $\cc=tp^{4n}+u-up^{2n+1}$.  If
$r(u)\mid p^m-1$ for some $m$ such that $1\le m\le 2n+1$
then $\OO_{K_{2n+1}}$ is free over its associated order
$\A_{K_{2n+1}/K_0}$.  In particular, if $r(u)=p^{2n+1}-1$
and $p(u+2)\le tp^{2n}$ then $\A_{K_{2n+1}/K_0}$ is a
Hopf order.  For instance, if $u=t=1$ and $e_0\ge
p^{2n}+1$, then $\OO_{K_{2n+1}}$ is free over
$\A_{K_{2n+1}/K_0}$, while if $u=p^{2n+1}-1$, $t=p^2+1$,
and $e_0\ge p^{2n+2}+p^{2n+1}+p^{2n}-1$ then
$\A_{K_{2n+1}/K_0}$ is a Hopf order.
\end{example}

\begin{remark}
Let $K$ be a local field of characteristic 0 with
residue characteristic $p$, let $G$ be a finite
noncyclic $p$-group, and let $\A$ be an $\OO_K$-order
in $K[G]$.  In Theorem~3.11 of \cite{BD09} it is proved
that the following are equivalent:
\begin{enumerate}[(i)]
\item $\A$ is a Hopf algebra over $\OO_K$ whose
$\OO_K$-dual $\A^*$ is a  local ring and a monogenic
$\OO_K$-algebra.
\item $\A$ contains no nontrivial idempotents and there
exists a totally ramified $G$-extension $L/K$ such that
the different of $L/K$ is generated by an element of
$K$, $\OO_L$ is free over $\A_{L/K}$, and the associated
order $\A_{L/K}$ of $\OO_L$ is isomorphic to $\A$.
\end{enumerate}
It follows that if $K_{2n+1}/K_0$ satisfies the
conditions of Theorem~\ref{GMS}(b) then
$\A_{K_{2n+1}/K_0}^*$ is a local ring and a monogenic
$\OO_0$-algebra.  In Section 4 of \cite{BD09} a method
is given for constructing $H(1)$-extensions $K_3/K_0$
which satisfy these equivalent conditions.
Example~\ref{ex} gives explicit constructions of
$H(n)$-extensions and $M(n)$-extensions which satisfy
these conditions.
\end{remark}

\bibliographystyle{amsalpha}
\bibliography{special}

\newcommand{\etalchar}[1]{$^{#1}$}
\providecommand{\bysame}{\leavevmode\hbox to3em{\hrulefill}\thinspace}
\providecommand{\MR}{\relax\ifhmode\unskip\space\fi MR }
% \MRhref is called by the amsart/book/proc definition of \MR.
\providecommand{\MRhref}[2]{%
  \href{http://www.ams.org/mathscinet-getitem?mr=#1}{#2}
}
\providecommand{\href}[2]{#2}
\begin{thebibliography}{CGK{\etalchar{+}}21}

\bibitem[Aib03]{Ai03}
Akira Aiba, \emph{Artin-{S}chreier extensions and {G}alois module structure}, J. Number Theory \textbf{102} (2003), no.~1, 118--124. \MR{1994476}

\bibitem[BBF72]{BBF72}
Fran\c{c}oise Bertrandias, Jean-Paul Bertrandias, and Marie-Jos\'{e}e Ferton, \emph{Sur l'anneau des entiers d'une extension cyclique de degr\'{e} premier d'un corps local}, C. R. Acad. Sci. Paris S\'{e}r. A-B \textbf{274} (1972), A1388--A1391. \MR{296048}

\bibitem[BCE18]{BCE18}
Nigel~P. Byott, Lindsay~N. Childs, and G.~Griffith Elder, \emph{Scaffolds and generalized integral {G}alois module structure}, Ann. Inst. Fourier (Grenoble) \textbf{68} (2018), no.~3, 965--1010. \MR{3805766}

\bibitem[BD08]{BD09}
M.~V. Bondarko and A.~V. Dievski\u{\i}, \emph{Nonabelian associated orders in the case of wild ramification}, Zap. Nauchn. Sem. S.-Peterburg. Otdel. Mat. Inst. Steklov. (POMI) \textbf{356} (2008), no.~Voprosy Teorii Predstavleni\u{\i} Algebr i Grupp. 17, 5--45, 189. \MR{2760364}

\bibitem[BE13]{BE13}
Nigel~P. Byott and G.~Griffith Elder, \emph{Galois scaffolds and {G}alois module structure in extensions of characteristic {$p$} local fields of degree {$p^2$}}, J. Number Theory \textbf{133} (2013), no.~11, 3598--3610. \MR{3084290}

\bibitem[BE18]{BE18}
\bysame, \emph{Sufficient conditions for large {G}alois scaffolds}, J. Number Theory \textbf{182} (2018), 95--130. \MR{3703934}

\bibitem[BF72]{BF72}
Fran\c{c}oise Bertrandias and Marie-Jos\'{e}e Ferton, \emph{Sur l'anneau des entiers d'une extension cyclique de degr\'{e} premier d'un corps local}, C. R. Acad. Sci. Paris S\'{e}r. A-B \textbf{274} (1972), A1330--A1333. \MR{296047}

\bibitem[Bon02]{Bon02}
M.~V. Bondarko, \emph{Local {L}eopoldt's problem for ideals in totally ramified {$p$}-extensions of complete discrete valuation fields}, Algebraic number theory and algebraic geometry, Contemp. Math., vol. 300, Amer. Math. Soc., Providence, RI, 2002, pp.~27--57. \MR{1936366}

\bibitem[CGK{\etalchar{+}}21]{Hopf}
Lindsay~N. Childs, Cornelius Greither, Kevin~P. Keating, Alan Koch, Timothy Kohl, Paul~J. Truman, and Robert~G. Underwood, \emph{Hopf algebras and {G}alois module theory}, Mathematical Surveys and Monographs, vol. 260, American Mathematical Society, Providence, RI, [2021] \copyright 2021. \MR{4390798}

\bibitem[Del84]{De84}
P.~Deligne, \emph{Les corps locaux de caract\'{e}ristique {$p$}, limites de corps locaux de caract\'{e}ristique {$0$}}, Representations of reductive groups over a local field, Travaux en Cours, Hermann, Paris, 1984, pp.~119--157. \MR{771673}

\bibitem[EK22]{EK22}
G.~Griffith Elder and Kevin Keating, \emph{Galois scaffolds for cyclic {$p^n$}-extensions in characteristic {$p$}}, Res. Number Theory \textbf{8} (2022), no.~4, Paper No. 75, 16. \MR{4491490}

\bibitem[EK23]{EKgen}
G.~Griffith {Elder} and Kevin {Keating}, \emph{{Galois scaffolds for $p$-extensions in characteristic $p$}}, arXiv e-prints (2023), arXiv:2308.02775, to appear in Ann. Inst. Fourier (Grenoble).

\bibitem[FV02]{FV}
I.~B. Fesenko and S.~V. Vostokov, \emph{Local fields and their extensions}, second ed., Translations of Mathematical Monographs, vol. 121, American Mathematical Society, Providence, RI, 2002, With a foreword by I. R. Shafarevich. \MR{1915966}

\bibitem[Gos96]{Go96}
David Goss, \emph{Basic structures of function field arithmetic}, Ergebnisse der Mathematik und ihrer Grenzgebiete (3) [Results in Mathematics and Related Areas (3)], vol.~35, Springer-Verlag, Berlin, 1996. \MR{1423131}

\bibitem[Hel90]{He90}
Charles Helou, \emph{Non-{G}alois ramification theory of local fields}, Algebra Berichte [Algebra Reports], vol.~64, Verlag Reinhard Fischer, Munich, 1990. \MR{1076620}

\bibitem[Hel91]{He91}
\bysame, \emph{On the ramification breaks}, Comm. Algebra \textbf{19} (1991), no.~8, 2267--2279. \MR{1123123}

\bibitem[KS22]{KS22}
Kevin Keating and Paul Schwartz, \emph{Galois scaffolds and {G}alois module structure for totally ramified {$C_{p^2}$}-extensions in characteristic 0}, J. Number Theory \textbf{239} (2022), 113--136. \MR{4434489}

\bibitem[Let05]{Le05}
G\"{u}nter Lettl, \emph{Note on a theorem of {A}. {A}iba}, J. Number Theory \textbf{115} (2005), no.~1, 87--88. \MR{2176484}

\bibitem[MW56]{MW56}
R.~E. MacKenzie and G.~Whaples, \emph{Artin-{S}chreier equations in characteristic zero}, Amer. J. Math. \textbf{78} (1956), 473--485. \MR{90584}

\bibitem[Noe32]{No32}
Emmy Noether, \emph{Normalbasis bei {K}\"{o}rpern ohne h\"{o}here {V}erzweigung}, J. Reine Angew. Math. \textbf{167} (1932), 147--152. \MR{1581331}

\bibitem[Ser79]{cl}
Jean-Pierre Serre, \emph{Local fields}, Graduate Texts in Mathematics, vol.~67, Springer-Verlag, New York-Berlin, 1979, Translated from the French by Marvin Jay Greenberg. \MR{554237}

\bibitem[VZ95]{VZ95}
S.~V. Vostokov and I.~B. Zhukov, \emph{Some approaches to the construction of abelian extensions for {${\mathfrak p}$}-adic fields}, Proceedings of the {S}t. {P}etersburg {M}athematical {S}ociety, {V}ol. {III}, Amer. Math. Soc. Transl. Ser. 2, vol. 166, Amer. Math. Soc., Providence, RI, 1995, pp.~157--174. \MR{1363296}

\end{thebibliography}

\end{document}